\documentclass{amsart}

\usepackage{amsmath,amsfonts,amssymb,amsthm,amscd,comment,euscript}
\usepackage[all]{xy}
\usepackage{graphicx}
\usepackage{enumerate} 
\usepackage[colorlinks=true]{hyperref} 
\usepackage[usenames,dvipsnames]{xcolor}
\usepackage[centering,margin=1in]{geometry}

\newcommand{\zz}{\mathbb{Z}}

\newcommand{\casetwo}[4]{\left\{ \begin{array}{ll} #1 &\mbox{if $#2$} \\[2mm] #3 &\mbox{if $#4$}\,. \end{array} \right.}

\theoremstyle{plain} 
\newtheorem{theo}{Theorem}[section]
\newtheorem{lem}[theo]{Lemma}
\newtheorem{prop}[theo]{Proposition}
\newtheorem{cor}[theo]{Corollary}
\newtheorem{conj}[theo]{Conjecture}
\newtheorem{prob}[theo]{Problem}

\theoremstyle{definition}
\newtheorem{example}[theo]{Example}
\newtheorem{defi}[theo]{Definition}
\newtheorem{examples}[theo]{Examples}

\theoremstyle{remark}
\newtheorem{rem}[theo]{Remark}
\newtheorem{rems}[theo]{Remarks}

\title[Schubert calculus via formal root polynomials]{Towards generalized cohomology Schubert calculus \\via formal root polynomials}

\author[C.~Lenart]{Cristian Lenart}
\address[Cristian Lenart]{Department of Mathematics and Statistics, State University of New York at Albany, 
Albany, NY 12222, U.S.A.}
\email{clenart@albany.edu}
\urladdr{http://www.albany.edu/\~{}lenart/}
\thanks{C.L. was partially supported by the NSF grants DMS--1101264 and DMS--1362627. He also gratefully acknowledges the hospitality and support of the Max-Planck-Institut f\"ur Mathematik in Bonn, where part of this work was carried out. K.Z. was partially supported by the NSERC Discovery grant  385795-2010 and the Early Researcher Award (Ontario). 
}

\author[K.~Zainoulline]{Kirill Zainoulline}
\address[Kirill Zainoulline]{Department of Mathematics and Statistics, University of Ottawa, 585 King Edward Street, Ottawa, ON, K1N 6N5, Canada}
\email{kirill@uottawa.ca}
\urladdr{http://mysite.science.uottawa.ca/kzaynull/}

\begin{document}
%

\begin{abstract}
An important combinatorial result in equivariant cohomology and $K$-theory Schubert calculus is represented by the formulas of Billey and Graham-Willems for the localization of Schubert classes at torus fixed points. These formulas work uniformly in all Lie types, and are based on the concept of a root polynomial. In this paper we define formal root polynomials associated with an arbitrary formal group law (and thus a generalized cohomology theory). We focus on the case of the hyperbolic formal group law (corresponding to elliptic cohomology). We study some of the properties of formal root polynomials. We give applications to the efficient computation of the transition matrix between two natural bases of the formal Demazure algebra in the hyperbolic case. As a corollary, we rederive in a transparent and uniform manner the formulas of Billey and Graham-Willems. We also prove the corresponding formula in connective $K$-theory, which seems new, and a duality result in this case. Other applications, including some related to the computation of Bott-Samelson classes in elliptic cohomology, are also discussed. 
\end{abstract}

\maketitle

\section{Introduction} Modern Schubert calculus has been mostly concerned with the cohomology and $K$-theory (as well as their quantum deformations) of generalized flag manifolds $G/B$, where $G$ is a connected complex semisimple Lie group and $B$ a Borel subgroup; Kac-Moody flag manifolds have also been studied, but we will restrict ourselves here to the finite case. Schubert calculus for other cohomology theories was first considered in \cite{BE90,BE92}, but the basic results have only been obtained recently in \cite{CZZ,CZZ1,CPZ,GaRa,hhhcge,HK,kakecf}. After this main theory has been developed, the next step is to give explicit formulas, thus generalizing well-known results in cohomology and $K$-theory, which are usually based on combinatorial structures. This paper is a contribution in this direction. 

We focus on elliptic cohomology, more precisely, on the associated hyperbolic formal group law, which we view as the first case after $K$-theory in terms of complexity. (The correspondence between generalized cohomology theories and formal group laws is explained below.) The main difficulty beyond $K$-theory is the fact that the naturally defined cohomology classes corresponding to a Schubert variety depend on its chosen Bott-Samelson desingularization, and thus on a reduced word for the given Weyl group element (which is not the case in ordinary cohomology and $K$-theory, where we have naturally defined Schubert classes). Thus the problem is to express combinatorially these so-called {\em Bott-Samelson classes}. The additional problem of defining Schubert classes independently of a reduced word will be considered elsewhere. 

In type $A$, there are well-known polynomial representatives for the Schubert classes in cohomology and $K$-theory, namely the {\em Schubert} and {\em Grothendieck polynomials} of Lascoux and Sch\"utzenberger \cite{lasps,lasagv}. We attempted to define similar representatives for Bott-Samelson classes in elliptic cohomology (as normal form representatives for the corresponding quotient ring), but we discovered that the natural elliptic generalizations of Schubert polynomials no longer have positive coefficients. Thus, it seems hard to extend to them, and to the similar generalizations of Grothendieck polynomials, the well-known combinatorial theory in \cite{BJS,Fo94,FK}. 

In this paper, we extend a different approach to combinatorial formulas for Schubert classes in (torus equivariant) cohomology and $K$-theory, which is due to Billey \cite{Bi99} and Graham-Willems \cite{graekt,wilcke}, respectively, and has the advantage of being uniform in all Lie types. Billey and Graham-Willems expressed the localizations of Schubert classes at torus fixed points, using the language in the fundamental papers of Demazure \cite{De74} and Kostant-Kumar \cite{KK86,KK90}. One can also use the language of {\em GKM theory} \cite{gkmeck}.  The combinatorics in \cite{Bi99,wilcke} is developed in the setup of root systems, and is based on the concept of a {\em root polynomial}. The goal of this paper is to define, study, and give applications of certain natural generalizations of root polynomials corresponding to an arbitrary formal group law, which we call {\em formal root polynomials}. 

We start by recalling the Kostant-Kumar and GKM setups for a generalized cohomology theory, in Section~\ref{background}; a central concept is that of the {\em formal Demazure algebra}. Then, in Section~\ref{frpp}, we define the formal root polynomials and study some of their main properties. In particular, since their definition depends on a reduced word for a Weyl group element, we show that, essentially, they are independent of this choice only up to the hyperbolic case (inclusive). We also explain the close connection to the hyperbolic solution of the Yang-Baxter equation \cite{Fr}. In Section~\ref{afrp}, we present our main application of formal root polynomials in the hyperbolic case; namely, we show that they provide an efficient way to compute the transition matrix between two natural bases of the formal Demazure algebra -- a problem which is implicit in \cite{CZZ1}. As a corollary, we are able to rederive in a transparent and uniform manner (i.e., in all Lie types, and for both ordinary cohomology and $K$-theory) the formulas of Billey and Graham-Willems. We also derive the corresponding formula in connective $K$-theory, which seems new, and show that it implies the generalization of the Kostant-Kumar duality result in $K$-theory. We  generalize other results in \cite{Bi99}, related to the Schubert structure constants. Finally, in Section~\ref{bscvrp}, we discuss the application of root polynomials to the computation of Bott-Samelson classes; in particular, we formulate two conjectures in the hyperbolic case, based on the analogy with the ordinary cohomology and $K$-theory cases, as well as on experimental evidence. 

\subsection*{Acknowledgments} We are grateful to W. Graham for sending us his preprint and for kindly explaining his results to us.

\section{Background}\label{background} We briefly recall the main results in Schubert calculus for generalized cohomology theories.

\subsection{Complex oriented cohomology theories} A (one dimensional, commutative) {\em formal group law} over a commutative ring $R$ is a formal power series $F(x,y)$ in $R[[x,y]]$ satisfying \cite{Haz}
\begin{equation} F(x,y)=F(y,x)\,,\;\;\;\;\;F(x,0)=x\,,\;\;\;\;\;F(x,F(y,z))=F(F(x,y),z)\,.\end{equation}
The {\em formal inverse} is the power series $\iota(x)$ in $R[[x]]$ defined by $F(x,\iota(x))=0$. The {\em exponential} of $F(x,y)$ and its compositional inverse, called the {\em logarithm}, are the power series ${\rm exp}_F(x)$ and ${\rm log}_F(x)$ in $R\otimes{\mathbb Q}[[x]]$ satisfying $F(x,y)={\rm exp}_F({\rm log}_F(x)+{\rm log}_F(y))$. 

Let $E^*(\cdot)$ be a complex oriented cohomology theory with base ring $R=E^*(\mbox{pt})$. By \cite{Qu71}, this is equipped with a formal group law $F(x,y)$ over $R$, which expresses the corresponding first Chern class, denoted $c(\cdot)$, of a tensor product of two line bundles ${\mathcal L}_1$ and ${\mathcal L}_2$ on a space $X$ in terms of $c({\mathcal L}_1)$ and $c({\mathcal L}_2)$; more precisely, we have 
\begin{equation}\label{defF} c({\mathcal L}_1\otimes {\mathcal L}_2)=F(c({\mathcal L}_1),\,c({\mathcal L}_2))\,.\end{equation}

We will now refer to the (finite type) generalized flag variety $G/B$, where $G$ is a complex semisimple Lie group, and we let $T$ be the corresponding maximal torus. We use freely the corresponding root system terminology, see e.g. \cite{Ku02}. As usual, we denote the set of roots by $\Phi$, the subsets of positive and negative roots by $\Phi^+$ and $\Phi^-$, the simple roots and corresponding simple reflections by $\alpha_i$ and $s_i$ (for $i=1,\ldots,n$, where $n$ is the rank of the root system), the lattice of integral weights by $\Lambda$, the fundamental weights by $\omega_i$, the Weyl group by $W$, its longest element by $w_\circ$, and its strong Bruhat order by $\le$. For each $w\in W$, we have the corresponding Schubert variety $X_w:=\overline{BwB}$. Given an arbitrary weight $\lambda$, let $\mathcal{L}_\lambda$ be the corresponding line bundle over $G/B$, that is, $\mathcal{L}_\lambda:= G\times_B {\mathbb C}_{-\lambda}$, where $B$ acts on $G$ by right multiplication, and the
$B$-action on ${\mathbb C}_{-\lambda}={\mathbb C}$ corresponds to the character determined by $-\lambda$. 
(This character of $T$ extends to $B$ by defining it to be 
identically one on the commutator subgroup $[B,B]$.) The relation \eqref{defF} becomes 
\begin{equation}\label{defFF} c({\mathcal L}_{\lambda+\nu})=F(c({\mathcal L}_{\lambda}),\, c({\mathcal L}_{\nu}))\,.\end{equation}

We now consider the respective $T$-equivariant cohomology $E_T^*(\cdot)$ of spaces with a $T$-action, see e.g. \cite{hhhcge}. Its base ring $E_T^*(\mbox{pt})$ can be identified (after the respective completion) with the {\em formal group algebra} 
\begin{equation}\label{defS}
S:=R[[y_\lambda]]_{\lambda\in \Lambda}/(y_0,\;
y_{\lambda+\nu}-F(y_{\lambda},y_{\nu}))
\end{equation}
of \cite[Def.~2.4]{CPZ}; in other words, $y_\lambda$ is identified with the corresponding first Chern class of ${\mathcal L}_\lambda$, cf. \eqref{defFF}. The Weyl group acts on $S$ by $w(y_\lambda):=y_{w\lambda}$. 

\begin{examples}\label{exfgl} (1) {\em Ordinary cohomology $H^*(\cdot)$ (with integer coefficients)}. The corresponding formal group law is the additive one $F_a(x,y)=x+y$, with formal inverse $\iota(x)=-x$. We identify $S$ with the completion of the symmetric algebra $Sym_\zz(\Lambda)$ via
$y_\lambda=-\lambda$. In fact, $Sym_\zz(\Lambda)\cong \zz[x_1,\ldots,x_n]$, where $x_i={\omega_i}$. We have $y_{-\lambda}=\iota(y_\lambda)=\lambda$.

(2) {\em $K$-theory $K(\cdot)$}. The corresponding formal group law is the (specialized) multiplicative one $F_m(x,y)=x+y-xy$, with formal inverse $\iota(x)=\frac{x}{x-1}$. We identify $S$ with the completion of the group algebra of the weight lattice $\zz[\Lambda]$; the latter has a $\zz$-basis of formal exponents $\{e^\lambda\,:\,\lambda\in\Lambda\}$, with multiplication $e^\lambda\cdot e^\nu=e^{\lambda+\nu}$. The identification is given by $y_\lambda=1-e^{-\lambda}$. In fact, $\zz[\Lambda]\cong \zz[x_1^{\pm 1},\ldots,x_n^{\pm 1}]$, where $x_i=e^{\omega_i}$. We have $y_{-\lambda}=\iota(y_\lambda)=1-e^{\lambda}$.

(3) {\em Connective $K$-theory}. The corresponding formal group law is the multiplicative one $F_m(x,y)=x+y-\mu_1 xy$, with $\mu_1$ not invertible in $R$. The formal inverse is $\iota(x)=\tfrac{x}{\mu_1x-1}$. The ring $S$ 
is the completion of the Rees ring $Rees(R[\Lambda],(\mu_1))$ \cite[\S4]{Hu12}. In
particular, 
\begin{equation}\label{connK}
(1-\mu_1y_{\lambda})(1-\mu_1 y_{\nu})=1-\mu_1y_{\lambda+\nu}\,.
\end{equation}

(4) {\em Elliptic cohomology $Ell^*(\cdot)$}. The corresponding formal group law is the group law
of an elliptic curve; see \cite{AHS01,MR07} and Section \ref{ellfgl}.

(5) {\em Complex cobordism $MU^*(\cdot)$}. The base ring $MU^*(\mbox{pt})$ is the {\em Lazard ring} ${\mathbb L}$, and the corresponding formal group law is the universal one. A celebrated result states that the Lazard ring is a polynomial ring (over the integers) in infinitely many variables; see \cite{Haz}, cf. also \cite{lensff}. 
\end{examples}

\begin{rem}\label{normalize}
Observe that if $F(x,y)=x+y+a_{11}xy+\ldots$, where $a_{11}\in
R^\times$, we can always normalize $F$, hence we can assume $a_{11}=-1$, cf. Examples~\ref{exfgl}~(2),(3). 
The normalization of formal group laws is discussed in detail in \cite{Zh}.
\end{rem}

As it was mentioned above, any complex oriented theory gives rise to a formal group law via the Quillen's observation. Unfortunately, the opposite fails in general: there are examples of formal group laws for which
the respective complex oriented cohomology theories simply do not exist. 
However, if one translates and extends the axiomatics of oriented theories 
into the algebraic context, which was done
by Levine and Morel in \cite{levmor-book}, then to any formal group law one can associate (by tensoring with algebraic cobordism over the Lazard ring) the respective \textit{algebraic} oriented cohomology theory.
Moreover, it was shown in \cite{CPZ} and \cite{CZZ}  that one can completely reconstruct the (algebraic) $T$-equivariant oriented cohomology ring of a flag variety starting with the formal group law and the root system only.

In the present paper we will work in this more general (algebraic) setting, assuming that $E^*_T(\cdot)$ stands for
the respective algebraic oriented cohomology.

\subsection{The hyperbolic formal group law}\label{ellfgl} 
We now introduce the key example of the present paper. Consider an elliptic curve given in Tate coordinates by 
$(1-\mu_1t-\mu_2t^2)s=t^3$. It was shown in \cite[Example 63]{BB10} and  \cite[Cor. 2.8]{BB11} that the respective formal group law is  
\begin{equation}\label{defell}
F(x,y)=\frac{x+y-\mu_1 xy}{1+\mu_2xy}\,,
\end{equation}
defined over $R=\zz[\mu_1,\mu_2]$. Note that the formal inverse is the same as for the multiplicative formal group law $F_m(x,y)=x+y-\mu_1 x y$, see Example~\ref{exfgl}~(3). 

This formal group law is, in fact, a very natural one from a topological perspective. It implicitly appeared  in Hirzebruch's celebrated book \cite{Hir} 
devoted to the proof of the Riemann-Roch theorem. Indeed, the book is centered on the notion of a
``(2-parameter) virtual generalized Todd genus''. It turns out that the associated formal group law (via the one-to-one correspondence between genera and formal group laws) is precisely the one in \eqref{defell}. The Hirzebruch genera and their relation to the mentioned formal group law were studied  in \cite{BB10,BB11}. 

Let us also mention some important special cases. 
In the trivial case $\mu_1=\mu_2=0$, $F(x,y)=F_a(x,y)$ is the additive group law, which corresponds to ordinary cohomology $H^*(\cdot)$.
In general, Hirzebruch sets $\mu_1=\alpha+\beta$  and $\mu_2=-\alpha\beta$. He observes that if $\mu_2=0$, then his genus is the usual Todd genus
(corresponding to $K$-theory, i.e., to the multiplicative formal group law $F_m(x,y)$). The case $\alpha=-\beta$, i.e., $\mu_1=0$, is related to the Atiyah-Singer signature, corresponding to the {\em Lorentz formal group law}. The case $\alpha=\beta$ corresponds to the Euler characteristic.  

The exponential of the formal group law $F(x,y)$ in \eqref{defell} is given by the rational function 
$$\exp_F(x)=\frac{e^{\alpha x}-e^{\beta x}}{\alpha e^{\alpha x}-\beta e^{\beta x}}$$ 
in $e^x$; this suggests that we call $F(x,y)$ a \textit{hyperbolic} formal group law and denote it by $F_h(x,y)$. Observe that $F_h(x,y)$ gives rise to a certain algebraic elliptic cohomology theory $Ell^*(\cdot)$. However,
if one sets $\alpha=\tfrac{t}{t+t^{-1}}$ and $\beta=\tfrac{t^{-1}}{t+t^{-1}}$ for $R=\mathbb{Z}[t,t^{-1},\tfrac{1}{t+t^{-1}}]$, then $F_h(x,y)$ does not correspond
to a complex oriented theory \cite[\S4]{BK91}; this case is important because the corresponding formal Demazure algebra (see Section~\ref{fda}) is isomorphic to the Hecke algebra.

For the theoretical foundations of elliptic Schubert calculus, we refer to \cite{Ga}.

\subsection{The structure of $E^*_T(G/B)$}\label{strucetgb} Given a Weyl group element $w$, consider a reduced word ${I_w}=(i_1,\ldots,i_l)$ for it, so $w=s_{i_1}\ldots s_{i_l}$. There is a {\em Bott-Samelson resolution} of the corresponding Schubert variety $X_w$, which we denote by $\gamma_{I_w}:\Gamma_{I_w}\rightarrow X_w\hookrightarrow G/B$. This determines a so-called {\em Bott-Samelson class} in  $E^*_T(G/B)$ via the corresponding pushforward map, namely $(\gamma_{I_w})_{!}(1)$. Here we let 
\begin{equation}\label{zetaw}\zeta_{I_w}:=(\gamma_{I_w^{-1}})_{!}(1)\,,\end{equation}
 where $I_w^{-1}:=(i_l,\ldots,i_1)$ is a reduced word for $w^{-1}$; we use $I_{w^{-1}}$, rather than $I_w$, in order to simplify the correspondence with the setup in \cite{Bi99,wilcke}, which is used in Sections~\ref{afrp} and \ref{bscvrp}. Note that $\zeta_{\emptyset}$ is the class of a point (where $\emptyset$ denotes the reduced word for the identity).

It is well-known \cite{BE90} that the Bott-Samelson classes are independent of the corresponding reduced words only for cohomology and $K$-theories (we can say that connective $K$-theory is the ``last'' case when this happens). In these cases, the Bott-Samelson classes are the {\em Schubert classes}, and they form bases of $H_T^*(G/B)$ and $K_T(G/B)$ over the corresponding ring $S$, as $w$ ranges over $W$. (More precisely, the Schubert classes are the Poincar\'e duals to the fundamental classes of Schubert varieties in homology, whereas in $K$-theory they are the classes of {\em structure sheaves} of Schubert varieties.) More generally, an important result in generalized cohomology Schubert calculus says that, by fixing a reduced word $I_w$ for each $w$, the corresponding Bott-Samelson classes $\{\zeta_{I_w}\,:\,w\in W\}$ form an $S$-basis of $E_T^*(G/B)$. 

There is a well-known model for $E_T^*(G/B)$ known as the {\em Borel model}, which we now describe. We start by considering the invariant ring 
\[S^W:=\{f\in S\,:\, wf=f \mbox{ for all $w\in W$}\}\,.\]
We then consider the coinvariant ring
\[S\otimes_{S^W} S:=\frac{S\otimes_R S}{\langle f\otimes 1-1\otimes f\,:\, f\in S^W\rangle}\,.\]
Here the product on $S\otimes_{S^W} S$ is given by $(f_1\otimes g_1)(f_2\otimes g_2):=f_1f_2\otimes g_1 g_2$. To more easily keep track of the left and right tensor factors, we set $x_\lambda:=1\otimes y_\lambda$ and $y_\lambda:=y_\lambda\otimes 1$. We use this convention whenever we work with a tensor product of two copies of $S$; by contrast, when there is a single copy of $S$ in sight, we let $x_\lambda=y_\lambda$. 
 
We are now ready to state a second important result in generalized cohomology Schubert calculus, namely that $S\otimes_{S^W} S$ is a rational model for $E_T^*(G/B)$, as an $S$-module; here the action of $y_\lambda\in S$ is on the left tensor factor, as the above notation suggests. 
Observe that, in general, $E_T^*(G/B)$ and  $S\otimes_{S^W} S$ are not isomorphic integrally (see \cite[Theorem~11.4]{CZZ}).

\subsection{The formal Demazure algebra}\label{fda} Following \cite[\S6]{HMSZ} and \cite[\S3]{CZZ}, consider the
localization $Q$ of $S$ along all $x_\alpha$, for $\alpha\in \Phi$ (note the change of notation, from $y_\lambda$ to $x_\lambda$, cf. the above convention),
and define the {\em twisted group algebra} $Q_W$  to be the smash product
$Q\# R[W]$, see \cite[Def.~6.1]{HMSZ}. More precisely, as an $R$-module, $Q_W$ is $Q\otimes_R R[W]$, while the
multiplication is given by
\begin{equation}\label{commrel}
q\delta_w \cdot q'\delta_{w'}=q(wq')\delta_{ww'},\qquad q,q'\in Q,\;\;
w,w'\in W.
\end{equation}

For simplicity, we denote
$\delta_i:=\delta_{s_i}$, $x_{\pm i}:=x_{\pm\alpha_i}$, and $x_{\pm i\pm j}:=x_{\pm\alpha_i\pm\alpha_j}$, for $i,j\in\{1,\ldots,n\}$; similarly for the $y$ variables.
Following \cite[Def.~6.2]{HMSZ} and \cite{CZZ1}, for each $i=1,\ldots, n$,
we define in $Q_W$
\begin{equation}\label{defxy}
X_i:=\frac{1}{x_i}\delta_i-\frac{1}{x_i}=\frac{1}{x_i}(\delta_i-1),\qquad
Y_i:=X_i+\kappa_i=\frac{1}{x_{-i}}+\frac{1}{x_i}\delta_i=(1+\delta_i)\frac{1}{x_{-i}}\,,
\end{equation}
where $\kappa_i$ is defined in \eqref{defkij}. We call $X_i$ and $Y_i$ the {\em Demazure} and the {\em push-pull element}, respectively. (Note that here, compared to \cite[Def.~6.2]{HMSZ} and \cite{CZZ1}, we changed the sign of $X_i$.) 
The $R$-algebra $\mathbf{D}_F$ generated by multiplication with elements of $S$ and the
elements $\{X_i\,:\,i=1,\ldots, n\}$, or 
$\{Y_i\,:\,i=1,\ldots, n\}$, is called the {\em formal affine Demazure algebra}. Observe that its dual $\mathbf{D}_F^\star$ serves as an integral model for $E_T^*(G/B)$.

The algebras $Q_W$ and $\mathbf{D}_F$ act on $S\otimes_R Q$ by
\begin{equation}\label{action}h(f\otimes g)=f\otimes hg\;\;\;\;\;\mbox{and}\;\;\;\;\;\delta_w(f\otimes g):=f\otimes wg \,,\end{equation}
where $f\in S$, $g,h\in Q$, and $w\in W$. In fact, the Demazure and push-pull elements act on $S\otimes_R S$ and $S\otimes_{S^W} S$.

\begin{examples}\label{exxy} We provide the explicit form of the action of $X_i$ and $Y_i$ on $S$ in several cases. In general, it is easiest to express the action of $X_i$ first, based on Examples~\ref{exfgl}, and then use $Y_i=X_i+\kappa_i$, with $\kappa_i$ calculated in Examples~\ref{exkij}.

(1) {\em Ordinary cohomology}. We have
\[X_i\,f=Y_i\,f=\frac{f-s_if}{\alpha_i}\,.\]

(2) {\em $K$-theory}. We have
\[X_i\,f=\frac{f-s_if}{e^{-\alpha_i}-1}\,,\qquad Y_i\,f=\frac{f-e^{\alpha_i}s_if}{1-e^{\alpha_i}}\,.\]

(3) {\em Connective $K$-theory}. We have
\[X_i\,f=\frac{f-s_if}{-x_i}\,,\qquad Y_i\,f=\frac{f-(1-\mu_1 x_{-i})s_if}{x_{-i}}\,.\]

(4) {\em The hyperbolic formal group law}. The same formulas as for connective $K$-theory apply.
\end{examples}

Let $m_{ij}$ be the order of $s_is_j$ in $W$, and define the following elements in $S$:
\begin{equation}\label{defkij}
\kappa_i:=\frac{1}{x_{-i}}+\frac{1}{x_i}\,,\qquad \kappa_{i,j}:=\frac{1}{x_{-i-j}}\left(\frac{1}{x_{-j}}-\frac{1}{x_{i}}\right)-\frac{1}{x_{-i}x_{-j}}\,.
\end{equation}
Note that the above definition of $\kappa_{i,j}$ differs from the one in \cite[Eq.~6.4]{HMSZ} by a sign change of both $i$ and $j$; the reason is that the mentioned paper gives the relations between the Demazure elements $X_i$, whereas here we need the relations between the push-pull elements $Y_i$, which we now discuss. 
According to \cite[Thm.~6.14]{HMSZ} and \cite[Prop.~8.10]{HMSZ}, the following relations hold in the algebra $\mathbf{D}_F$.
\begin{itemize}
\item[(a)] For all $i$, we have
\begin{equation}\label{square}
Y_i^2=\kappa_i Y_i\,.\end{equation}  
\item[(b)] If $\langle\alpha_i,\alpha_j^\vee\rangle=0$, so that $m_{ij}=2$, then 
\begin{equation}\label{braid2}Y_iY_j=Y_jY_i\,.
\end{equation}
\item[(c)] If $\langle\alpha_i,\alpha_j^\vee\rangle=\langle\alpha_j,\alpha_i^\vee\rangle=-1$, so that $m_{ij}=3$, then 
\begin{equation}\label{braid3}Y_iY_jY_i-Y_jY_iY_j=\kappa_{j,i}Y_j-\kappa_{i,j}Y_i\,.\end{equation} 
\item[(d)] If $\langle\alpha_i,\alpha_j^\vee\rangle=-2$ and $\langle\alpha_j,\alpha_i^\vee\rangle=-1$, so that $m_{ij}=4$, then 
\begin{align}\label{braid4}Y_iY_jY_iY_j-Y_jY_iY_jY_i=&(\kappa_{j,i}+\kappa_{i+2j,-j})Y_jY_i-(\kappa_{i,j}+\kappa_{i+j,j})Y_iY_j\\
&+\left(Y_i(\kappa_{i,j}+\kappa_{i+j,j})\right)Y_j-\left(Y_j(\kappa_{j,i}+\kappa_{i+2j,-j})\right) Y_i\,.
\nonumber\end{align}
\end{itemize}
There is also a more involved relation if $m_{ij}=6$, given in \cite[Prop.~6.8~(d)]{HMSZ}. This relation as well as  \eqref{braid3} and \eqref{braid4} are called {\em twisted braid relations}. 

\begin{examples}\label{exkij} (1) {\em The additive formal group law}. We have $\kappa_i=0$, so $Y_i^2=0$. Moreover, all $\kappa_{i,j}$ are $0$, so the twisted braid relations are the usual braid relations. 

(2) {\em The multiplicative formal group law $F_m(x,y)=x+y-\mu_1xy$}. We have $\kappa_i=\mu_1$ for all $i$, so $Y_i^2=\mu_1 Y_i$. On another hand, all $\kappa_{i,j}$ are again $0$. 

(3) {\em The hyperbolic formal group law {\rm \eqref{defell}}}. It is easy to see that we still have $\kappa_i=\mu_1$ for all $i$. On another hand, based on the calculation \eqref{calckij} below, we have $\kappa_{i,j}=\mu_2$ for all $i,j$.   So \eqref{braid3} and \eqref{braid4} become the following ones, respectively:
\begin{align}
&Y_iY_jY_i-Y_jY_iY_j=\mu_2(Y_j-Y_i)\,,\label{braid3ell}\\
&Y_iY_jY_iY_j-Y_jY_iY_jY_i=2\mu_2(Y_jY_i-Y_iY_j)\,.
\label{braid4ell}
\end{align}
Moreover, the parameters $\xi_{i,j}$ in the twisted braid relation for $m_{ij}=6$, given in \cite[Prop.~6.8~(d)]{HMSZ}, is equal to $3(\mu_2)^2$, see \cite{LNZ}.
\end{examples}

Let us justify that $\kappa_{i,j}=\mu_2$ in the case of the hyperbolic formal group law. By plugging $x_{i}=\iota(x_{-i})$ and $x_{-i-j}=F_m(x_{-i},x_{-j})$ into \eqref{defkij}, cf. Example~\ref{exfgl}~(3), we have
\begin{equation}\label{calckij}
\kappa_{i,j}=\frac{1+\mu_2x_{-i}x_{-j}}{x_{-i}+x_{-j}-\mu_1x_{-i}x_{-j}}\left(\frac{1}{x_{-j}}-\frac{\mu_1x_{-i}-1}{x_{-i}}\right)-\frac{1}{x_{-i}x_{-j}}=\mu_2\,.
\end{equation}

Given a reduced word $I_w=(i_1,\ldots,i_l)$ for $w\in W$, define $X_{I_w}:=X_{i_1}\ldots X_{i_l}$ and $Y_{I_w}:=Y_{i_1}\ldots Y_{i_l}$. 
By \cite{CZZ1}, if we fix a
reduced word $I_w$ for each $w\in W$, then $\{X_{I_w}\,:\,w\in W\}$ and $\{Y_{I_w}\,:\,w\in W\}$ are bases of the free left
$Q$-module $Q_W$. Note that, in cohomology and $K$-theory, $X_{I_w}$ and $Y_{I_w}$ do not depend on the choice of the reduced word $I_w$ (see Examples~\ref{exkij}~(1),(2)), so we can simply write $X_w$ and $Y_w$. 

A fundamental result in generalized cohomology Schubert calculus states that the Bott-Samelson classes $\zeta_{I_w}$, for ${I_w}=(i_1,\ldots,i_l)$, can be calculated recursively as follows:
\begin{equation}\label{pushpull}
\zeta_{I_w}=Y_{i_l}\ldots Y_{i_1}\,\zeta_{\emptyset}\,.
\end{equation}
where ${\zeta}_{\emptyset}$ is the class of a point ($\zeta_{\emptyset}=x_\Pi f_e$ in the notation of \cite{CZZ1}). By analogy with \eqref{pushpull}, we define the following classes:
\begin{equation}\label{pushpulld}\widetilde{\zeta}_{I_w}:=X_{i_l}\ldots X_{i_1}\,\zeta_{\emptyset}\,;\end{equation}
clearly, these coincide with $\zeta_{I_w}$ for ordinary cohomology.

Given a choice of reduced words $I_w$ for $w\in W$, let $X_{I_w}^*$ and $Y_{I_w}^*$ in the dual Demazure algebra $\mathbf{D}_F^\star$ be the usual duals of $X_{I_w}$ and $Y_{I_w}$, respectively. By \cite[Theorem~12.4]{CZZ1}, $Y_{I_w}^*$ is also dual to $\zeta_{I_w}$  with respect to a standard Poincar\'e-type pairing defined in the mentioned paper; the same is true for $X_{I_w}^*$ and $\widetilde{\zeta}_{I_w}$.

\subsection{The GKM model of equivariant cohomology}\label{gkmsetup}
This setup is summarized in \cite[Theorem~3.1]{GaRa}, which cites \cite{hhhcge, kakecf, CPZ}, see also the seminal paper \cite{gkmeck}. 

In the GKM model, we embed $E_T^*(G/B)$ into $\bigoplus_{w\in W}S$, with pointwise multiplication. 
This comes from the embedding 
\begin{equation}\label{gkmemb}i^*\,:\,E_T^*(G/B)\rightarrow\bigoplus_{w\in W}E_T^*({\rm{pt}})\simeq\bigoplus_{w\in W}S\,,\end{equation}
 where 
\begin{equation}\label{defi}i^*:=\bigoplus_{w\in W} i_w^*\,,\;\;\;\;\mbox{and}\;\;\;\;i_w\,:\,\mbox{pt}\rightarrow G/B\,,\:\mbox{ with $\mbox{pt}\mapsto w^{-1}$}\,.\end{equation}
There is a characterization of the image of this embedding, see e.g. \cite{GaRa}. Note that, in \eqref{defi}, like in the definition \eqref{zetaw} of Bott-Samelson classes, we diverge from the usual definitions in \cite{GaRa} by interchanging the roles of $w$ and $w^{-1}$; the reason is the same: to simplify the correspondence with the setup in \cite{Bi99,wilcke}, which is used in Sections~\ref{afrp} and \ref{bscvrp}. We denote the elements of $\bigoplus_{w\in W}S$ by $(f_w)_{w\in W}$; alternatively, we view them as functions $f:W\to S$. 

Using the Borel model for $E_T^*(G/B)$, we can realize the  GKM map $i^*$ in \eqref{gkmemb} as an embedding of $S\otimes_{S^W} S$ into $\bigoplus_{w\in W} S$. This map can be made explicit as
\begin{equation}\label{inc2}f\otimes g\mapsto (f\cdot(w g))_{w\in W}\,.\end{equation}
Via this map, the action \eqref{action} of the algebras $Q_W$ and $\mathbf{D}_F$ is translated as follows in the GKM model, cf. \cite[Eq.~(3.19),~(3.20)]{GaRa}:
\begin{equation}\label{actiongkm} x_\lambda\cdot 1=(x_{w\lambda})_{w\in W}\,,\qquad \delta_v\,(f_w)_{w\in W}=(f_{wv})_{w\in W}\,.\end{equation}

As now the action of the push-pull operators $Y_i$ is made explicit, we can use \eqref{pushpull} to compute recursively the Bott-Samelson classes $\zeta_{I_w}$ in the GKM model, once we know the class $\zeta_{\emptyset}$. This is given by
\begin{equation}\label{topclass}
(\zeta_{\emptyset})_w=\casetwo{\prod_{\alpha\in\Phi^+}y_{-\alpha}}{w={\rm Id}}{0}{w\ne{\rm Id}}
\end{equation}
In fact, the following more general result holds:
\begin{equation}\label{topclassgen}
(\zeta_{I_v})_w=\casetwo{\prod_{\alpha\in \Phi^+\cap w\Phi^+}y_{-\alpha}}{v=w}{0}{w\not\le v}
\end{equation}

\section{Formal root polynomials and their properties}\label{frpp}

In this section we define {\em formal root polynomials} -- the main object of this paper, -- by extending the definitions in \cite{Bi99,wilcke} in a natural way, and begin their study. The setup is the one above, so it corresponds to an arbitrary formal group law.

\subsection{Definition and basic facts}\label{deffacts}

Consider the ring
$Q_W[[\Lambda]]_F=S\otimes_R Q_W$, where the elements of $S$ on the left (denoted by $y$'s) commute with the elements
of $Q_W$. The formal root polynomials are elements of this ring which depend on a reduced word $I_w=(i_1,\ldots,i_l)$ for a Weyl group element $w$. In fact, we will define two formal root polynomials, corresponding to the Demazure and push-pull operators $X_i$ and $Y_i$ in \eqref{defxy}. It is well-known that $I_w$ induces
a so-called {\em reflection order} on the roots in $\Phi^+\cap w\Phi^-$, namely
\begin{equation}\label{reflord}
\Phi^+\cap w\Phi^-=\{\alpha_{i_1},\,s_{i_1}\alpha_{i_2},\,\ldots,\,s_{i_1}\ldots s_{i_{l-1}}\alpha_{i_l}\}\,.
\end{equation}

\begin{defi}\label{defroot} The formal $Y$-root polynomial corresponding to $I_w$ is
\[
\mathcal{R}^Y_{I_w}:=\prod_{k=1}^{l} h_{i_k}^Y(s_{i_1}\ldots
s_{i_{k-1}}\alpha_{i_k}),\qquad \mbox{where }\,h_i^Y(\lambda)=1-y_{\lambda}Y_i\,.
\]
Similarly, we define the corresponding $X$-root polynomial by
\[
\mathcal{R}^X_{I_w}:=\prod_{k=1}^{l} h_{i_k}^X(s_{i_1}\ldots
s_{i_{k-1}}\alpha_{i_k}),\qquad \mbox{where }\,h_i^X(\lambda)=1+y_{-\lambda}X_i\,.
\]
\end{defi}

\begin{example}\label{exrpa2} For type $A_2$ and $w=s_is_js_i$, $I_w=(i,j,i)$, $i\neq
  j$, we obtain (cf. the notation in Section \ref{fda})
\begin{align*}
{\mathcal R}^Y_{(i,j,i)}
&=(1-y_iY_i)(1-y_{i+j}Y_j)(1-y_jY_i)\\&=-y_i y_{i+j}y_j
Y_iY_jY_i+y_iy_{i+j}Y_iY_j+y_jy_{i+j}Y_jY_i+(\kappa_iy_iy_j-y_i-y_j)
Y_i -y_{i+j} Y_j +1\,.
\end{align*}
\end{example}

Consider a ring homomorphism ${\rm ev}\colon S\otimes_R Q \to Q$ given by
$y_\lambda\mapsto x_{-\lambda}$, for $\lambda\in \Lambda$.
Since $Q_W[[\Lambda]]_F$ is a left $S\otimes_R Q$-module, it induces a homomorphism of left $S\otimes_R
Q$-modules ${\rm ev}\colon Q_W[[\Lambda]]_F\to
Q_W$, which we call the {\em evaluation map}.

\begin{lem}\label{evallem}
We have
\[
{\rm ev}(\mathcal{R}^X_{I_w})=\delta_w\,, \qquad {\rm ev}(\mathcal{R}^Y_{I_w})=\theta_{I_w} \delta_w\,, \quad \mbox{where }\,\theta_{I_w}:=\prod_{k=1}^{l} \theta(s_{i_1}\ldots
s_{i_{k-1}}\alpha_{i_k})\,,\;\;
\theta(\lambda)=\frac{-x_{-\lambda}}{x_\lambda}\in S\,.
\]
Furthermore, $\theta_{I_w}$ does not depend on the choice of the reduced word $I_w$ for $w$, so we write $\theta_w=\theta_{I_w}$.
\end{lem}

\begin{proof}
Clearly, $\theta_{I_w}$ only depends on the set $\Phi^+\cap w\Phi^-$, and thus only on $w$.
We prove the desired formula for $\mathcal{R}^Y_{I_w}$ by induction on the length of $w$, while the proof for $\mathcal{R}^X_{I_w}$ is completely similar.
For length $1$ we have
\[
{\rm ev}(\mathcal{R}^Y_i)={\rm ev}(1-y_{i}Y_i)=1-x_{-i}\left(\frac{1}{x_{-i}}+\frac{1}{x_i}\delta_i\right)=\frac{-x_{-i}}{x_i}\delta_i=\theta_i\,\delta_i\,.
\]
Now consider $w'\in W$ and $I_{w'}$, ending in $i$. Let $w:=w's_i$, with $\ell(w)=\ell(w')-1$, and let $I_w$ be obtained by removing the last index in $I_{w'}$. By induction, we assume that ${\rm ev}(\mathcal{R}^Y_{I_w})=\theta_w\,\delta_w$. 
By using \eqref{commrel}, we have
\begin{align*}
{\rm ev}(\mathcal{R}^Y_{I_{w'}})&={\rm ev}(\mathcal{R}^Y_{I_w} 
(1-y_{w\alpha_i}Y_i))={\rm ev}\left(\mathcal{R}^Y_{I_w}-y_{w\alpha_i}\mathcal{R}^Y_{I_w}\left(\frac{1}{x_{-i}}+\frac{1}{x_i}\delta_i\right)\right)\\
&=\theta_w\,\delta_w-x_{-w\alpha_i}\,\theta_w\,\delta_w
\left(\frac{1}{x_{-i}}+\frac{1}{x_i}\delta_i\right)=\theta_w\,\delta_w -
\theta_w\,\delta_w\,
x_{-i}\left(\frac{1}{x_{-i}}+\frac{1}{x_i}\delta_i\right)\\
&=\theta_w\,\delta_w \,\theta(\alpha_i) \,\delta_i=\theta_w\,\theta(w\alpha_i)\,\delta_w\, \delta_i=\theta_{ws_i}\,\delta_{ws_i}\,. \qedhere
\end{align*}
\end{proof}

We will call $\theta_w$ the {\em normalizing parameter}. We now consider some special cases, cf. Examples~\ref{exfgl}~(1)--(4), including the conventions therein.

\begin{examples}\label{exeval} (1) {\em Ordinary cohomology}. 
We have $h_i^Y(\lambda)=1+\lambda Y_i$ and $Y_i^2=0$, so the corresponding root
polynomials coincide with the ones of Billey \cite[Definition 2]{Bi99}.
We have $\theta(\lambda)=1$, and therefore ${\rm ev}(\mathcal{R}^Y_{I_w})=\delta_w$.

(2) {\em $K$-theory}. 
We have $h_i^Y(\lambda)=1-(1-e^{-\lambda}) Y_i=1+(e^{-\lambda}-1)Y_i$ and $Y_i^2=Y_i$, so the corresponding root
polynomials coincide with the ones of Willems \cite[Section 5]{wilcke}. 

We have $\theta(\lambda)=\tfrac{1}{1-x_\lambda}=\tfrac{1}{1-(1-e^{-\lambda})}=e^\lambda$. Therefore, the normalizing parameter can be expressed as
\[
\theta_w=e^{\sum_{k=1}^l s_{i_1}\ldots
s_{i_{k-1}}\alpha_{i_k}}=e^{\rho-w\rho}\,,
\]
where $\rho=\tfrac{1}{2}\sum_{\alpha\in\Phi^+}\alpha$, as usual; indeed, $\rho-w\rho$ is the sum of roots in the set $\Phi^+\cap
w\Phi^-$.

(3) {\em Connective $K$-theory}. Recall that, in this case, $\mu_1$ in $F_m(x,y)=x+y-\mu_1xy$ is not invertible in $R$. 
We have
$\theta(\lambda)=\tfrac{1}{1-\mu_1 x_\lambda}$. Therefore, for rank $n$ root systems, we can express the normalizing parameter as follows, based on \eqref{connK}: 
\begin{equation}\label{thetack}
\theta_w=\prod_{k=1}^{l} \frac{1}{1-\mu_1x_{s_{i_1}\ldots
s_{i_{k-1}}\alpha_{i_k}}}=\prod_{i=1}^n\frac{1}{(1-\mu_1x_i)^{m_i}}\,,
\end{equation}
where $m_i$ is the sum of $\alpha_i$-coefficients of the (positive)
roots in $\Phi^+\cap w\Phi^-$.

In particular, let $R=\zz[t]$ be a polynomial ring (with
variable $\mu_1=t$). 
For type $A_2$ with simple roots $\alpha_i$ and $\alpha_j$, we obtain
\[
\theta_{(i,j,i)}=\frac{1}{(1-tx_i)(1-tx_{i+j})(1-tx_j)}=\frac{1}{(1-tx_i)^2(1-tx_j)^2}\,.
\]

(4) {\em The hyperbolic formal group law {\rm \eqref{defell}}}. Since the formal inverse is the same as for the multiplicative formal group law, the normalizing parameter is also given by \eqref{thetack}. If $\mu_1=0$ (giving the Lorentz formal group law), then the case of the additive formal group law applies, i.e., Example~\ref{exeval}~(1), so $\theta_w=1$ and ${\rm ev}(\mathcal{R}^Y_{I_w})=\delta_w$.
\end{examples}

\subsection{Independence of choices of reduced words}\label{indredw}
We now provide necessary and sufficient
conditions for the root polynomials $\mathcal{R}^Y_{I_w}$ and $\mathcal{R}^X_{I_w}$ to be independent
of a choice of $I_w$, 
hence generalizing the respective results of Billey \cite{Bi99} and  Graham-Willems \cite{graekt,wilcke}. For simplicity, we refer only to the $Y$-root polynomial, and let $h_i=h_i^Y$; however, all the results hold for the $X$-root polynomial as well.

\begin{prop}\label{ybecomm} If $\langle\alpha_i,\alpha_j^\vee\rangle=0$, i.e., $m_{ij}=2$, then we have
\[
h_i(\lambda)\,h_j(\nu)=h_j(\nu)\,h_i(\lambda)\text{ for all }\lambda,\nu\in \Lambda\,,
\]
i.e., the $h_i$'s satisfy the commuting relations.
\end{prop}

\begin{proof}
We explicitly compute
\[
h_i(\lambda)\,h_j(\nu)=(1-y_{\lambda}Y_i)(1-y_{\nu}Y_j)=1-y_{\lambda}Y_i-y_{\nu}Y_2+y_{\lambda}y_{\nu}Y_iY_j\,.
\]
Since $Y_iY_j=Y_jY_i$, the result follows.
\end{proof}

\begin{prop}\label{ybea} Consider a root system containing a pair of simple roots $\alpha_i,\alpha_j$ with $m_{ij}=3$, which we fix. Also assume that $2$ is not a zero divisor in the coefficient ring $R$. Then the following
  conditions are equivalent.
\begin{itemize}
\item[(a)]
The underlying formal group law is the hyperbolic one {\rm \eqref{defell}}.  
\item[(b)] We have
\begin{equation}\label{allw}h_i(\lambda)\,h_{j}(\lambda+\nu)\,h_i(\nu)=h_{j}(\nu)\,h_i(\lambda+\nu)\,h_{j}(\lambda)\,,\end{equation}
for all $\lambda,\nu\in \Lambda$.
In this case, we say that $h_i$ and $h_j$ satisfy the (type $A$) Yang-Baxter relation for
all weights.
\item[(c)] We have
$$h_i(\alpha_i)\,h_{j}(\alpha_i+\alpha_{j})\,h_i(\alpha_{j})=h_{j}(\alpha_{j})\,h_i(\alpha_i+\alpha_{j})\,h_{j}(\alpha_i)\,.$$
In this case, we say that $h_i$ and $h_j$
satisfy the (type $A$) Yang-Baxter relation for the corresponding simple roots.
\end{itemize}
\end{prop}

\begin{proof}
We have
\begin{align*}
&h_i(\lambda)\,h_{j}(\lambda+\nu)\,h_i(\nu)-h_{j}(\nu)\,h_i(\lambda+\nu)\,h_{j}(\lambda)
\\
=&(1-y_{\lambda}Y_i)(1-y_{\lambda+\nu}Y_j)(1-y_{\nu}Y_i)-(1-y_{\nu}Y_j)(1-y_{\lambda+\nu}Y_i)(1-y_{\lambda}Y_j)\,.
\end{align*}
Applying the relations \eqref{square} and \eqref{braid3} in $\mathbf{D}_F$, this can be rewritten as follows:
\begin{align*}
&(y_{\lambda+\nu}-y_{\lambda}-y_{\nu})(Y_i-Y_j)+y_{\lambda}y_{\nu}(\kappa_iY_i-\kappa_jY_j)+y_{\lambda}y_{\lambda+\nu}y_{\nu}(Y_jY_iY_j-Y_iY_jY_i)
\\
=&\left(y_{\lambda+\nu}(1+y_{\lambda}y_{\nu}\kappa_{i,j})-(y_{\lambda}+y_{\nu}-y_{\lambda}y_{\nu}\kappa_i)\right)Y_i\\
&-\left(y_{\lambda+\nu}(1+y_{\lambda}y_{\nu}\kappa_{j,i})-(y_{\lambda}+y_{\nu}-y_{\lambda}y_{\nu}\kappa_j)\right)Y_j\,.
\end{align*}

Recall that, if we choose a
reduced word $I_v$ for each $v\in W$, then $\{Y_{I_v}\,:\,v\in W\}$ is a basis of the free left
$Q$-module $Q_W$. In particular, it is also a basis of the free left
$S\otimes_R Q$-module $Q_W[[\Lambda]]_F$. Since $Y_i$ and $Y_j$ are
basis elements (there is only one choice of a reduced word for an
element of length one), $h_i$ and $h_j$ satisfy the Yang-Baxter relation for all weights if and only
we have
\begin{equation}\label{expr}
y_{\lambda+\nu}(1+y_{\lambda}y_{\nu}\kappa_{i,j})-(y_{\lambda}+y_{\nu}-y_{\lambda}y_{\nu}\kappa_i)=0\quad\text{
for all }\lambda,\nu\in \Lambda.
\end{equation}

(a) $\Longrightarrow$ (b). 
We saw in Example~\ref{exkij}~(3) that, if $F(x,y)$ is the hyperbolic formal group law \eqref{defell}, then $\kappa_i=\mu_1$ and $\kappa_{i,j}=\mu_2$. 
So \eqref{expr} turns into
\[
y_{\lambda+\nu}(1+\mu_2y_{\lambda}y_{\nu})-(y_{\lambda}+y_{\nu}-\mu_1y_{\lambda}y_{\nu})=0,
\]
and, therefore, $h_i$ and $h_j$ satisfy the Yang-Baxter relation for all weights.

(b) $\Longrightarrow$ (c) is obvious.

(c) $\Longrightarrow$ (a). 
Assume that $h_i$ and $h_j$ satisfy the Yang-Baxter relation for the corresponding simple roots,
i.e., that we have
\[
y_{i+j}(1+y_{i}y_{j}\kappa_{i,j})-(y_{i}+y_{j}-y_{i}y_{j}\kappa_i)=0\,,
\]
where $\kappa_i$ and $\kappa_{i,j}$ are the expressions in $x$'s defined in \eqref{defkij}. 
Substituting
$y_{i+j}=y_i+y_j+a_{11}y_iy_j+a_{12}y_iy_j(y_i+y_j)+o(4)$ and
collecting the coefficients, 
we obtain
\[
0=(a_{11}+\kappa_i)y_iy_j +(\kappa_{i,j}+a_{12})y_iy_j(y_i+y_j)+o(4)\,.
\]
Therefore, $\kappa_i=-a_{11}$ and $\kappa_{i,j}=-a_{12}$ have to be constants.
So $F(y_i,y_j)(1-a_{12}y_iy_j)-(y_i+y_j+a_{11}y_iy_j)=0$ and, hence,
$F(x,y)$ is the hyperbolic formal group law (indeed, $y_i$ and $y_j$ are independent variables).
\end{proof}

\begin{rems} (1) Proposition~\ref{ybea} shows that, in the formal Kostant-Kumar language of
  \cite{HMSZ,CZZ,CZZ1}, the Yang-Baxter relation 
corresponds precisely to the hyperbolic formal group law. It also justifies the study of
such formal group laws.

(2) The Yang-Baxter relation satisfied by $h_i$ and $h_j$ for the corresponding simple roots was called by Stembridge the {\em Coxeter-Yang-Baxter relation} \cite{stecyb}. In this unpublished paper, he was able to prove it in a type-independent way in the case of the additive formal group law. His approach does not seem to extend to the hyperbolic case, so we need to consider types $A_2$, $B_2$, and $G_2$ separately, cf. Proposition~\ref{ybeb} below. However, see Remark~\ref{remind}~(2).
\end{rems}

Now recall from \cite{Bi99,wilcke,stecyb} the Yang-Baxter relations of types $B$ and $G$. 

\begin{prop}\label{ybeb} Consider a root system containing a pair of simple roots $\alpha_i,\alpha_j$ with $m_{ij}=4$ or $m_{ij}=6$, which we fix. Also assume that the underlying formal group law is the hyperbolic one. Then $h_i$ and $h_j$ satisfy the (type $B$, resp. type $G$)  Yang-Baxter relation for all weights, cf. {\rm \eqref{allw}}; in type $B$, this means
\begin{equation}\label{byb}h_i(\lambda-\nu)\,h_j(\lambda)\,h_i(\lambda+\nu)\,h_j(\nu)=
h_j(\nu)\,h_i(\lambda+\nu)\,h_j(\lambda)\,h_i(\lambda-\nu)\,,\end{equation}
for all $\lambda,\nu\in \Lambda$.
\end{prop}

\begin{proof}
Let $i,j$ be such that $\langle\alpha_i,\alpha_j^\vee\rangle=-2$ and $\langle\alpha_j,\alpha_i^\vee\rangle=-1$. We can rewrite \eqref{byb} as
\begin{align*}
&(1-y_{\lambda-\nu}Y_i)(1-y_{\lambda}Y_j)(1-y_{\lambda+\nu}Y_i)(1-y_{\nu}Y_j)-\\
&-(1-y_{\nu}Y_j)(1-y_{\lambda+\nu}Y_i)(1-y_{\lambda}Y_j)(1-y_{\lambda-\nu}Y_i)=0\,.
\end{align*}
It is enough to show that the coefficients of $Y_iY_j$ and $Y_jY_i$ in the left hand side cancel. We will focus on $Y_iY_j$, as the case of $Y_jY_i$ is identical. 

We will use the fact that, by \eqref{square} and Example~\ref{exkij}~(3), for the hyperbolic formal group law we have $Y_i^2=\mu_1Y_i$, $Y_j^2=\mu_1 Y_j$, while the twisted braid relations are \eqref{braid4ell}. 
Hence, the desired coefficient of $Y_iY_j$ is calculated as follows, after grouping terms appropriately:
\[y_{\lambda+\nu}(y_\nu-y_\lambda)+y_{\lambda-\nu}(y_\lambda+y_\nu-\mu_1y_\lambda y_\nu)-\mu_1y_\nu y_{\lambda+\nu}y_{\lambda-\nu}-2\mu_2y_\lambda y_\nu y_{\lambda+\nu}y_{\lambda-\nu}\,.\]
We use the following relation between the $y$ variables (expressed in terms of the formal group law, see \eqref{defS}) in order to rewrite the second bracket:
\[y_\lambda+y_\nu-\mu_1y_\lambda y_\nu=y_{\lambda+\nu}(1+\mu_2 y_\lambda y_\nu)\,.\]
After doing this, dividing through by $y_{\lambda+\nu}$, and then canceling and regrouping terms, we obtain
\[(y_{\lambda-\nu}+y_\nu-\mu_1y_{\lambda-\nu} y_\nu)-y_\lambda(1+\mu_2 y_{\lambda-\nu} y_\nu)\,.\]
But this is equal to $0$, by using the relations between the $y$ variables a second time. 

The case $m_{ij}=6$ was checked with the help of a computer.
\end{proof}

\begin{rem} Let us comment on the reciprocal of the statement in Proposition~\ref{ybeb} in type $B$, cf. Proposition~\ref{ybea}. Following the same argument as in the proof of the implication (c) $\Longrightarrow$ (a) in the latter proposition (in type $A$), but after more extensive computations, we can show that, if $h_i$ and $h_j$ satisfy the type $B$  Yang-Baxter relation for the corresponding simple roots, then the corresponding formal group law $F(x,y)=\sum_{i,j} a_{ij}x^iy^j$ satisfies
\[\kappa_i=\kappa_j=-a_{11}\,,\qquad\kappa_{i+j,j}+\kappa_{ij}=\kappa_{i+2j,-j}+\kappa_{ji}=-2 a_{12}\,.\]
This fact then gives the following condition for $F(x,y)$, where $[2]y:=F(y,y)$:
\[(-1+2a_{12}xy)F(x,y)F(x,[2]y) + y(1+a_{11}x)F(x,[2]y) + x(1+a_{11}y)F(x,y) + xy = 0\,.\]
It would be interesting to see if there are other formal group laws beside the hyperbolic one which satisfy this condition and, if so, whether the corresponding type $B$ Yang-Baxter relation holds. 
\end{rem}

\begin{theo}\label{ind}
The root polynomial $\mathcal{R}^Y_{I_w}$ does not depend on the choice of
$I_w$ if the underlying formal group law $F(x,y)$ is the hyperbolic one; so we can write ${\mathcal R}^Y_w$ instead. The reciprocal holds if the corresponding root system contains a pair of simple roots $\alpha_i,\alpha_j$ with $m_{ij}=3$.
\end{theo}

\begin{proof}
If $F(x,y)$ is the hyperbolic formal group law, then the $h_i$'s satisfy the commuting and the
Yang-Baxter relations, by Propositions \ref{ybecomm}, \ref{ybea}, \ref{ybeb}. Therefore, the proof can use the same arguments as the proofs of \cite[Theorem~2]{Bi99}, \cite[Theorem~2]{stecyb}, and \cite[Theorem~5.2]{wilcke}. Essentially, we use the connectivity of the graph on reduced words for $w$ given by Coxeter relations, and for each such relation we apply the corresponding commuting or Yang-Baxter relation.

Now assume that the corresponding root system contains a pair of simple roots $\alpha_i,\alpha_j$ with $m_{ij}=3$, and that $\mathcal{R}^Y_{I_w}$ does not depend on the choice of $I_w$. In particular, we have
\[
\mathcal{R}^Y_{(i,j,i)}-\mathcal{R}^Y_{(j,i,j)}=h_i(\alpha_i)\,h_j(\alpha_i+\alpha_j)\,h_i(\alpha_j)-h_j(\alpha_j)\,h_i(\alpha_i+\alpha_j)\,h_j(\alpha_i)=0\,,
\]
that is, $h_i$ and $h_j$ satisfy the Yang-Baxter relation for the corresponding simple roots. By Proposition~\ref{ybea}, this is equivalent to $F(x,y)$ being the hyperbolic formal group law.
\end{proof}

\begin{rems}\label{remind} (1) The first part of Theorem~\ref{ind} appears as \cite[Theorem~2]{Bi99} and \cite[Theorem~2]{stecyb}, resp. \cite[Theorem~5.2]{wilcke}, in the case of the additive formal group law, resp. the multiplicative one.

(2) An alternative proof of the first part of Theorem~\ref{ind}, which is uniform for all Lie types, is given in Section \ref{coeffroot}.
\end{rems}

\section{Applications of formal root polynomials in the hyperbolic case}\label{afrp}

\subsection{The coefficients of formal root polynomials}\label{coeffroot}

In this section we present our main application of formal root polynomials in the hyperbolic case, which is continued in the next section.

We start with an arbitrary formal group law $F(x,y)$. 
Consider the root polynomials $\mathcal{R}^Y_{I_w}$ and $\mathcal{R}^X_{I_w}$ as elements of $Q_W[[\Lambda]]_F$.
Recall that, after we fix a reduced word $I_v$ for each $v\in W$, 
the elements $\{Y_{I_v}\,:\,v\in W\}$ and $\{X_{I_v}\,:\,v\in W\}$ form bases of the left $S\otimes_R
Q$-module $Q_W[[\Lambda]]_F$.

\begin{defi}\label{defkos}
Let $K^Y(I_v,I_w)$, resp. $K^X(I_v,I_w)$, denote the coefficient of $Y_{I_v}$ in the expansion of  $\mathcal{R}^Y_{I_w}$, resp. the coefficient of $X_{I_v}$ in the expansion of  $\mathcal{R}^X_{I_w}$, i.e.
\begin{equation}\label{rwexp}
\mathcal{R}^Y_{I_w}=\sum_{v\le w} K^Y(I_v,I_w)\, Y_{I_v}\,,\qquad \mathcal{R}^X_{I_w}=\sum_{v\le w} K^X(I_v,I_w)\, X_{I_v}\,.
\end{equation}
By analogy with the setup of {\rm \cite{Bi99}}, we call $K^Y(I_v,I_w)$ and $K^X(I_v,I_w)$ {\em (normalized) formal Kostant polynomials}.
\end{defi}

From now on, we assume that we are in the hyperbolic case, so by Theorem \ref{ind} we know that the root polynomials do not depend on the choice of a reduced word. However, we will not use this knowledge in this section, so we will continue to use the notation $\mathcal{R}^Y_{I_w}$ and $\mathcal{R}^X_{I_w}$.
By using the twisted braid relations in the hyperbolic case, see Example~\ref{exkij}~(3), we can express $\mathcal{R}^Y_{I_w}$  as a linear
combination of $Y_{I_v}$ with coefficients in $S$ (i.e., in the image
of $S\hookrightarrow 
S\otimes_R Q$). Therefore, $K^Y(I_v,w) \in S$, as an expression only
in the $y$-variables. Similarly for $\mathcal{R}^X_{I_w}$ and $K^X(I_v,w)$. (In general, the twisted braid relations introduce $x$-variables through the coefficients $\kappa_{i,j}$ and $\xi_{i,j}$, see Section \ref{fda}.)

\begin{examples}\label{exbilwil} (1) {\em Ordinary cohomology}. The coefficients $K^Y(I_v,I_w)$ do not depend on the choices of $I_v$ and $I_w$, and we can write Billey's result \eqref{billey-main} below in the following form (see Theorem~\ref{locck}~(1)):
\begin{equation}
\psi^v(w)=K^Y(v,w)\,.
\end{equation}
Essentially, $K^Y(v,w)$ also coincide with the (modified) Kostant polynomials $\widetilde{K}_v({\mathbf O}w)$ in Billey's paper \cite{Bi99}.

(2) {\em $K$-theory}. The coefficients $K^Y(I_v,I_w)$ still do not depend on choices, and we can write Willems' result \eqref{willems-main} below in the following form (see Theorem~\ref{locck}~(1)):
\begin{equation}
\psi^v(w)=\theta_w\,K^Y(v,w)=e^{\rho-w\rho} \,K^Y(v,w)\,.
\end{equation} 
\end{examples}

Observe that the free left $S\otimes_R Q$-module $Q_W[[\Lambda]]_F$ also has a 
basis $\{\delta_w\,:\,w\in W\}$.
Consider the following expansions, cf. \cite[Lemmas~3.1~and~3.2]{CZZ1}:
\begin{equation}\label{expdw}\delta_w=\sum_{v\le
  w}b_{w,I_v}^Y\,Y_{I_v}\,,\qquad \delta_w=\sum_{v\le
  w}b_{w,I_v}^X\,X_{I_v}\end{equation}
   for some coefficients $b_{w,I_v}^Y$ and $b_{w,I_v}^X$ in $S$ (in
$x$'s).

\begin{example} We have 
\begin{equation}\label{bvv}
b_{v,I_v}^Y =b_{v,I_v}^X= \prod_{\alpha \in \Phi^+ \cap v\Phi^-} x_\alpha\,,
\end{equation}
by \cite[Lemmas~3.1~and~3.2]{CZZ1}. Observe that these coefficients do not depend on the choice of $I_v$, so we can denote them by $b_{v,v}^Y=b_{v,v}^X$.
Also note that, in the case of the additive formal group law, these coefficients coincide with the normalizing factor $\pi_v({\mathbf O})$
in Billey's paper \cite{Bi99}. 
\end{example}

By applying the evaluation map ${\rm ev}$ to \eqref{rwexp} and by using Lemma~\ref{evallem}, we obtain
\[
\theta_w\,\delta_w=\sum_{v\le w}{\rm ev}(K^Y(I_v,I_w)) \,Y_{I_v}\,.
\]
Comparing with \eqref{expdw}, we conclude that 
\[b_{w,I_v}^Y=\theta_w^{-1} \,{\rm ev}(K^Y(I_v,I_w))\,;\]
observe that $\theta_w^{-1}\in S$. Also note that, since $K^Y(I_v,I_w)$ is expressed only in terms of $y$-variables, the evaluation map simply changes them to $x$-variables via $y_\lambda\mapsto x_{-\lambda}$. Recall our convention to identify $x_\lambda$ with $y_\lambda$ if there is a single copy of $S$ in sight. Also, we define the sign change involution $*$ on $S$ by $y_\lambda\mapsto y_{-\lambda}$, and we extend it in the obvious way to $Q_W$, i.e., $*$ fixes $\delta_w$. As $\theta_w^{-1}=*\theta_w$, we have proved the following result, while a completely similar reasoning proves its analogue for the $X$-root polynomial.

\begin{theo}\label{mainthm} In the hyperbolic case, we have in $S$:
\[
b_{w,I_v}^Y=*(\theta_w\, K^Y(I_v,I_w))\,,\qquad b_{w,I_v}^X=* K^X(I_v,I_w) \,.\]
In particular, $K^Y(I_v,I_w)$ (resp. $K^X(I_v,I_w)$), and hence ${\mathcal R}_{I_w}^Y$ (resp. ${\mathcal R}_{I_w}^X$), do not depend on the choice of $I_w$, so we can use the notation $K^Y(I_v,w)$ and ${\mathcal R}_{w}^Y$ (resp. $K^X(I_v,w)$ and ${\mathcal R}_{w}^X$) in this case. 
\end{theo}

\begin{rems} (1) The formula in Theorem \ref{mainthm} relates the coefficients of the $Y$-root polynomial with the
  coefficients of the transformation matrix between two natural bases of the
  twisted formal group algebra; similarly for the $X$-root polynomial. This provides 
 an efficient way to compute these matrices, as shown in Examples~\ref{exbvv} below; note that a direct computation is difficult, because the Leibniz rule needs to be applied repeatedly. By contrast, the inverse matrices, expressing $Y_{I_v}$ and $X_{I_v}$ in terms of $\delta_w$, can be easily computed based on definitions: write $Y_i$ and $X_i$ as in \eqref{defxy} and expand. For cohomology and $K$-theory, the entries of the two pairs of inverse matrices feature prominently in the work of Kostant-Kumar \cite{KK90}. They are related by a duality property (which fails in the hyperbolic case), and they encode information about the singularities of Schubert varieties, see also \cite[Chapter~7]{BL}; thus, it is expected that, in the hyperbolic case, $b_{w,I_v}^Y$ is a more refined invariant of singularities. 

 (2) In this section, we have given an alternative proof to the first part of Theorem \ref{ind}, which has the advantage of being transparent and type-independent. Furthermore, in the next section we will show that the localization formulas of Billey and Graham-Willems easily follow from Theorem \ref{mainthm}, when specialized to the cases of the additive and multiplicative formal group laws. So we provide a transparent and uniform proof (i.e., in all Lie types, and for both ordinary cohomology and $K$-theory) of the mentioned results. In particular, in our approach the fact that the root polynomial is independent of a reduced word is a consequence, rather than an input to the main proof. 

(3) Theorem~\ref{mainthm} cannot hold (in its present form) beyond the hyperbolic case. Indeed,  by the second part of Theorem \ref{ind}, we know that, for a root polynomial to be independent of a reduced word, we must be in the hyperbolic case (this statement is true as long as the root system contains a pair of simple roots connected by a single edge in the Dynkin diagram). However, it would be interesting to find a generalization of Theorem~\ref{mainthm} beyond the hyperbolic case. 
 \end{rems}

\begin{examples}\label{exbvv} (1) We can easily rederive the formula for $b_{v,v}^Y$ in \eqref{bvv}. Indeed, it follows from Definition \ref{defroot}, by also taking into account the form of the twisted braid relations,  that
$$K^Y(I_{v},v)=(-1)^{\ell(v)}\prod_{\alpha\in \Phi^+\cap v\Phi^-}x_\alpha\,.$$ 
Then we combine this with the definition of the normalizing parameter $\theta_v$ in Lemma~\ref{evallem}, namely
\[\theta_v=(-1)^{\ell(v)}\prod_{\alpha\in \Phi^+\cap v\Phi^-}\frac{x_{-\alpha}}{x_\alpha}\,,\]
and apply the sign change involution to the product of $\theta_v$ and $K^Y(I_{v},v)$. In a completely similar way, we derive the identical formula for $b_{v,v}^X$ in \eqref{bvv}.

(2) Consider the root system of type $A_2$ and the Lorentz formal group law, i.e.,  \eqref{defell} with $\mu_1=0$. For each permutation $v\ne w_\circ$ in $S_3$ we have a unique reduced word $I_v$, and choose $I_{w_\circ}=(1,2,1)$. Based on ${\mathcal R}^Y_{(1,2,1)}$ calculated in Example~\ref{exrpa2}, while also recalling that in this case $\kappa_i=0$ (see Example~\ref{exkij}~(3)), $\theta_w=1$ (see Example~\ref{exeval}~(4)), and $y_{-\alpha}=-y_\alpha$, we have
\[\delta_{w_\circ}=y_1 y_{1+2}y_2
Y_1Y_2Y_1+y_1y_{1+2}Y_1Y_2+y_2y_{1+2}Y_2Y_1+(y_1+y_2)
Y_1 +y_{1+2} Y_2 +1\,.\]
If we choose $I_{w_\circ}=(2,1,2)$ instead, we need to adjust the above expansion by rewriting $Y_1Y_2Y_1$ as $Y_2Y_1Y_2+\mu_2(Y_2-Y_1)$ via \eqref{braid3ell}, and then collecting terms. Alternatively, we  can avoid using the twisted braid relation by using the expression for ${\mathcal R}^Y_{(2,1,2)}$ in Example~\ref{exrpa2}. Note that, although ${\mathcal R}^Y_{(1,2,1)}={\mathcal R}^Y_{(2,1,2)}$, the expressions for them given in Example~\ref{exrpa2} are different. 
\end{examples}

\subsection{The structure constants of the dual Demazure algebra} 
Let us now turn to a generalization of the cohomology and $K$-theory Schubert structure constants. Given a choice of reduced words $I_w$, recall the usual dual $Y_{I_w}^*$ of $Y_{I_w}$ in the dual Demazure algebra $\mathbf{D}_F^\star$, see Section~\ref{fda}. Let
\[Y_{I_u}^* \cdot Y_{I_v}^* =\sum_{w\ge u,v}p_{I_u,I_v}^{I_w}Y_{I_w}^*\] 
be the expansion of the given product into the basis $\{Y_{I_w}^*\}$ of $\mathbf{D}_F^\star$. We follow the approach in Billey's paper \cite{Bi99} related to the cohomology structure constants. We start by generalizing \cite[Lemma~5.3]{Bi99}.

\begin{prop}
The coefficients $p_{I_u,I_v}^{I_w}$ can be computed recursively
by
\[
p_{I_u,I_v}^{I_w}=\frac{1}{b_{w,I_w}^Y}\Big(b_{w,I_u}^Y b_{w,I_v}^Y-\sum_{t<w}p^{I_t}_{I_u,I_v}b_{w,I_t}^Y\Big)\,,
\]
starting with $w=u$ (for $v\le u$, $p_{I_u,I_v}^{I_u}=b_{u,I_v}^Y$) and going up in length.
\end{prop}

\begin{proof}
In the notation of \cite{CZZ} and \cite{CZZ1} $Y_{I_v}^*=\sum_{w\ge v}b^Y_{w,I_v}f_w$ (here $f_w\in Q_W^*$ is the usual dual to $\delta_w$). Therefore, $Y_{I_v}^*(\delta_w)=b^Y_{w,I_v}$. Observe that $Y_{I_v}^*(\delta_w)=0$ if $\ell(w)\le \ell(v)$ and $w\neq v$.
Therefore, the expansion of any element $f \in Q_W^*$ into the basis $\{Y_{I_w}^*\}$ is given by $f=\sum_{w\in W}c_{I_w}Y_{I_w}^*$, where
\[
c_{I_w}=\frac{1}{b_{w,I_w}^Y}\Big(f(\delta_w)-\sum_{v<w}c_{I_v}Y_{I_v}^*(\delta_w)\Big)\,.
\]
Applying this to the product $f=Y_{I_u}^* \cdot Y_{I_v}^*$, we get
\[
p_{I_u,I_v}^{I_w}=Y_{I_u}^*(\delta_w) \cdot Y_{I_v}^*(\delta_w)=\frac{1}{b_{w,I_w}^Y}\Big(Y_{I_u}^*(\delta_w)\cdot Y_{I_v}^*(\delta_w)-\sum_{t<w}p^{I_t}_{I_u,I_v} b_{w,I_t}^Y\Big)\,.\qedhere
\]
\end{proof}

\begin{cor}
For any formal group law, the coefficients $p_{I_u,I_v}^{I_w}$ do not depend on the choice of $I_w$, so we can write $p_{I_u,I_v}^w=p_{I_u,I_v}^{I_w}$.
\end{cor}

Other results in Billey's paper \cite{Bi99} can be generalized as well. Proposition~5.4 generalizes as follows: the coefficients $p_{I_u,I_v}^w$ are in $\mathcal{I}^{\ell(u)+\ell(v)-\ell(w)}$ (the power of the augmentation ideal of $S$). Proposition~5.5 is the same, except for replacing $d_{uv}$ by $b^Y_{v,I_u}$. Hence, in the hyperbolic case, via Theorem~\ref{mainthm}, we are able to compute $p_{I_u,I_v}^w$ in terms of the coefficients $K^Y(I_v,w)$ of $Y$-root polynomials. This gives another application of these polynomials. Furthermore, we have completely similar results for the structure constants in the multiplication of the basis elements $\{X_{I_w}^*\}$. 

\section{Localization formulas in cohomology and $K$-theory}\label{local}

\subsection{The Kostant-Kumar functions} The goal of Billey's paper \cite{Bi99}, related to ordinary cohomology, and of the papers of Graham \cite{graekt} and Willems \cite{wilcke}, related to $K$-theory, is to express the functions $\psi^v$ and $\xi^v$ from $W$ to $S$ defined by Kostant and Kumar in \cite{KK86} and \cite{KK90}, respectively. In this section we recall these functions, and relate them to the classes $\zeta_v$ in \eqref{pushpull}. Note that Billey and Willems interchange the roles of $v$ and $v^{-1}$ in all the formulas of Kostant-Kumar, which we also do. 

Recall from Examples~\ref{exfgl}~(1),(2) the special form of $S$ in cohomology and $K$-theory, which will be implicitly used. 
In the cohomology case, the functions $\psi^v$ are characterized by the following conditions:
\begin{enumerate}
\item[(H1)] $\psi^v(w)$ equals $0$ unless $v\le w$, and $\psi^w(w)=\prod_{\alpha\in\Phi^+\cap w\Phi^-}\alpha$;
\item[(H2)] given the operators $A_i$ (associated to the simple roots) acting on functions $f:W\to S$ by
\begin{equation}\label{defai}(A_i\,f)(w)=\frac{f(w)-f(ws_i)}{-w\alpha_i}\,,\end{equation}
we have
\begin{equation}\label{recai}
A_i\,\psi^v=\casetwo{\psi^{vs_i}}{vs_i<v}{0}{vs_i>v}
\end{equation}
\end{enumerate} 

Similarly, in $K$-theory, the functions $\psi^v$ are characterized by the following conditions:
\begin{enumerate}
\item[(K1)] $\psi^v(w)$ equals $0$ unless $v\le w$, and $\psi^w(w)=\prod_{\alpha\in\Phi^+\cap w\Phi^-}(1-e^\alpha)$;
\item[(K2)] given the operators $D_i$ (associated to the simple roots) acting on functions $f:W\to S$ by
\begin{equation}\label{defdi}(D_i \,f)(w)=\frac{f(w)-e^{-w\alpha_i}f(ws_i)}{1-e^{-w\alpha_i}}\,,\end{equation}
we have
\begin{equation}\label{recdi}
D_i\,\psi^v=\casetwo{\psi^v+\psi^{vs_i}}{vs_i<v}{0}{vs_i>v}
\end{equation}
\end{enumerate} 

The $K$-theoretic functions $\xi^v\,:\,W\to S$ are defined by analogy with $\psi^v$. More precisely, condition (K1) is identical, while condition (K2), namely \eqref{recdi}, is replaced with the following one, for the same operators $D_i$:
\begin{equation}\label{recdi1}
D_i\,\xi^v=\casetwo{\xi^{vs_i}}{vs_i<v}{\xi^v}{vs_i>v}
\end{equation}

Let us now discuss the relationship between the functions $\xi^v$ and $\psi^v$ and the corresponding classes ${\zeta}_{I_v}={\zeta}_v$ and $\widetilde{\zeta}_{I_v}=\widetilde{\zeta}_v$ in  \eqref{pushpull} and \eqref{pushpulld}. Recall that in ordinary cohomology we have $X_i=Y_i$, whereas in $K$-theory we have $X_i=Y_i-1$. We establish the mentioned relationship by using the corresponding GKM model, so we represent the elements of $E_T(G/B)$ as functions $f\colon W\to S$. 
Given a function $f:W\to S$, let $f^{\circ}:W\to S$ be defined by $f^{\circ}(w):=(\delta_{w_\circ}(f))(w)=f(ww_\circ)$, cf. \eqref{actiongkm}. Also recall the involution $\alpha_i\mapsto -w_\circ\alpha_i=:\alpha_{i^*}$ on the simple roots. 

The key fact is the following relation for functions $f:W\to S$:
\begin{equation}\label{keyrel}
(X_{{i}^*} \,f^{\circ})^{\circ}=X_{-i}\, f\,,
\end{equation}
where $X_{-i}:=*X_i$, see the definition of the sign change involution in Section \ref{coeffroot}. 
Indeed, based on the action formula \eqref{actiongkm} in the GKM embedding \eqref{inc2}, the operator $X_{\pm i}$ is expressed in the GKM model as
\[
(X_{\pm i}\,f)(w)=\frac{1}{x_{\pm w \alpha_i}} (f(ws_i)-f(w))\,.
\]
Thus we have
\[
X_{i^*}\,f^{\circ}(w)=\frac{1}{x_{w\alpha_i^*}}(f^{\circ}(ws_{i^*})-f^{\circ}(w))=\frac{1}{x_{-ww_\circ\alpha_i}}(f(ws_{i^*}w_\circ)-f(ww_\circ)).
\]
Since $s_{i^*}=w_\circ s_i w_\circ$, the claim \eqref{keyrel} follows.

Letting $Y_{-i}:=*Y_i$, observe that the operator $A_i$ for ordinary cohomology coincides with the operator $X_{-i}=Y_{-i}$
and the operator
$D_i$ for $K$-theory coincide with the operator $Y_{-i}=1+X_{-i}$.
Indeed, for ordinary cohomology this follows from Example~\ref{exxy}~(1) and the fact that $x_\lambda$ corresponds to $-\lambda$.
And for $K$-theory it follows from Example~\ref{exxy}~(2) and the identification $x_{-\lambda}=1-e^\lambda$.


\begin{prop}\label{wil-gen}
{\rm (1)} In cohomology and in $K$-theory, we have 
\begin{equation}\label{zetaphi}\widetilde{\zeta}_{vw_\circ}(ww_\circ)=\psi^{v}(w)\,,\;\;\;\;\; i.e.,\;\;(\widetilde{\zeta}_{vw_\circ})^{\circ}=\psi^v\,.\end{equation}

{\rm (2)} In $K$-theory, we have 
\begin{equation}\label{zetaxi}{\zeta}_{vw_\circ}(ww_\circ)=\xi^{v}(w)\,,\;\;\;\;\; i.e.,\;\;({\zeta}_{vw_\circ})^{\circ}=\xi^v\,.\end{equation}
\end{prop}

\begin{proof}
For the first part, we proceed by induction with base case \eqref{topclass} and conditions (H1)-(K1) above, and induction step based on \eqref{recai}, \eqref{recdi}, \eqref{pushpulld} and \eqref{keyrel}: if $vs_i<v$, then
\[
\psi^{vs_i}=X_{-i}\,\psi^v=(X_{i^*}\,(\psi^v)^{\circ})^{\circ}=(X_{i^*}\,\widetilde{\zeta}_{vw_\circ})^{\circ}=(\widetilde{\zeta}_{vs_iw_\circ})^{\circ}\,. 
\]
The second part is completely similar.
\end{proof}

Let us now consider the functions $*\psi^v$, cf. the notation introduced in Section~\ref{coeffroot}. Since $\psi^v$ are generated recursively by $X_{-i}$, as discussed above, and since $X_i\,(*f)=*(X_{-i}\,f)$, it follows that
\begin{equation}\label{recdual}
*\psi^v=X_{(w_\circ v)^{-1}}\,(*\psi^{w_\circ})\,.
\end{equation}

\begin{rem}\label{geometry} Based on the discussion in Sections \ref{strucetgb} and \ref{gkmsetup}, as well as on Proposition~\ref{wil-gen}~(2), the evaluation $\xi^v(w)$ can be interpreted geometrically as $i_{ww_\circ}^*([{\mathcal O}_{vw_\circ}])$, where ${\mathcal O}_u$ is the structure sheaf of the Schubert variety indexed by $u$. On another hand, the Poincar\'e dual of the $K$-theory class $\zeta_v=[{\mathcal O}_v]$ is the class of the {\em ideal sheaf} of the Schubert variety indexed by $v$, which under $i_w^*$ is mapped to $*\psi^v(w)$, see e.g. \cite{KK90}, \cite[Prop.~2.1]{gakpte}.
The duality property, which holds for both cohomology and $K$-theory, 
says that there is a non-degenerate pairing 
\begin{equation}\label{idealsh-dual}
p_!(\zeta_v \cdot *\psi^w)=p_!(Y_{v^{-1}}\,\zeta_\emptyset \cdot X_{(w_\circ v)^{-1}}\,(*\psi^{w_\circ}))=\delta_{v,w}\,,
\end{equation}
where $p_!$ is the degree map for cohomology and the Euler characteristic for $K$-theory.
\end{rem}

\subsection{The formulas of Billey and Graham-Willems}\label{bwres} 
We are now stating the combinatorial formulas for the functions $\psi^v$ (in cohomology and $K$-theory) and $\xi^v$ (in $K$-theory) due to Billey and  Graham-Willems. 

Recall from Examples~\ref{exkij}~(1),(2) that the push-pull operators $Y_i$ satisfy the relations $Y_i^2=0$ in cohomology, $Y_i^2=Y_i$ in $K$-theory, and the usual braid relations in both cases. In other words, the algebras generated by them are the {\em nil-Coxeter algebra} and the {\em $0$-Hecke algebra}, respectively; we consider these algebras over $S$. Recall that they have bases $\{Y_w\,:\,w\in W\}$. 


Consider a reduced word $I_w=(i_1,\ldots,i_l)$ for $w$, and recall the corresponding reflection order \eqref{reflord} on $\Phi^+\cap w\Phi^-$, namely $(\beta_1,\ldots,\beta_l)$ with $\beta_j:=s_{i_1}\ldots s_{i_{j-1}}\alpha_{i_j}$. Then, as stated in \cite[Theorem 4]{Bi99}, in the cohomology case we have
\begin{equation}\label{billey-main}
\psi^v(w)=\sum_{1\le j_1<\ldots< j_k\le l}\beta_{j_1}\ldots\beta_{j_k}\,,
\end{equation}
where $k=\ell(v)$, and the summation ranges over the integer sequences $1\le j_1<\ldots< j_k\le l$ for which the subword  $(i_{j_1},\ldots,i_{j_k})$ of $I_w$ is a reduced word for $v$.

The corresponding result in $K$-theory, stated as \cite[Theorem~3.7]{graekt} and as \cite[Theorem~4.7]{wilcke}, is the following:
\begin{equation}\label{willems-main}
\psi^v(w)=e^{\rho-w\rho}\sum_{1\le j_1<\ldots< j_k\le l}(e^{-\beta_{j_1}}-1)\ldots(e^{-\beta_{j_k}}-1)\,,
\end{equation}
where $\ell(v)\le k\le l$, and the summation ranges over the integer sequences $1\le j_1<\ldots< j_k\le l$ for which the corresponding subword $(p_1,\ldots,p_k)=(i_{j_1},\ldots,i_{j_k})$ of $I_w$ is subject to $Y_{p_1}\ldots Y_{p_k}=Y_v$ in the $0$-Hecke algebra.

We now state Graham's formula for $\xi^v$, which appears as \cite[Theorem~3.12]{graekt}. To this end, we note that the Demazure operators $X_i$ in $K$-theory satisfy $X_i^2=-X_i$, and we can consider the basis $\{X_w\,:\,w\in W\}$ of the $0$-Hecke algebra. With the same notation as above, we have
\begin{equation}\label{willems-main1}
\xi^v(w)=\sum_{1\le j_1<\ldots< j_k\le l}(-1)^{k-\ell(v)}\,(1-e^{\beta_{j_1}})\ldots(1-e^{\beta_{j_k}})\,,
\end{equation}
where $\ell(v)\le k\le l$, and the summation ranges over the integer sequences $1\le j_1<\ldots< j_k\le l$ for which the corresponding subword $(p_1,\ldots,p_k)=(i_{j_1},\ldots,i_{j_k})$ of $I_w$ is subject to $X_{p_1}\ldots X_{p_k}=\pm X_v$ in the $0$-Hecke algebra.

\begin{rems}
(1) We derived \eqref{willems-main1} in our setup independently (see Section~\ref{bwrev}), before W. Graham sent us his preprint.

(2) In \cite[Theorem~3.7]{graekt}, Graham gives the formula for $*\psi^v(w)$ (cf.~Remark~\ref{geometry}), which can be easily obtained from \eqref{willems-main} by changing the signs of all exponents in the right-hand side.
\end{rems}

\subsection{Localization formulas in connective $K$-theory}\label{seclocck} We work in the setup of connective $K$-theory, see Example \ref{exfgl}~(3).  We set $t:=\mu_1$, such that the cases $t=0$, resp. $t=1$, correspond to ordinary cohomology, resp. $K$-theory. Recall from Example~\ref{exkij}~(2) that the operators $Y_i$ and $X_i$ satisfy $Y_i^2=tY_i$,  $X_i^2=-tX_i$, and the usual braid relations. Like above, we consider the bases $\{Y_w\,:\,w\in W\}$ and $\{X_w\,:\,w\in W\}$ of the corresponding Demazure algebra.

In concordance with Proposition~\ref{wil-gen}, we define
\begin{equation}\label{defpsiv}
\psi^v:=(\widetilde{\zeta}_{vw_\circ})^{\circ}\,,\;\;\;\;\;\xi^v:=({\zeta}_{vw_\circ})^{\circ}\,.
\end{equation}
By a completely analogous proof to that of Proposition~\ref{wil-gen}, we prove that 
\begin{equation}\label{recpsixi}X_{-i}\,\psi^v=\psi^{vs_i}\,,\;\;\;\;\;Y_{-i}\,\xi^v=\xi^{vs_i}\,,\;\;\;\;\mbox{for $vs_i<v$}\,;
\end{equation}
in particular, the version of \eqref{recdual} in connective $K$-theory (having the same form) holds, cf. Remark~\ref{geometry}. Recall the root polynomials ${\mathcal R}^Y_w$ and ${\mathcal R}^X_w$ (cf. Theorem~\ref{mainthm}), as well as the formal Kostant polynomials $K^Y(v,w)$ and $K^X(v,w)$ in Definition~\ref{defkos} (corresponding to our connective $K$-theory case). The following result is the main one of this section. It generalizes the similar results in cohomology and $K$-theory in \cite{Bi99,wilcke}; note that the latter immediately imply \eqref{billey-main} and \eqref{willems-main}, see below.

\begin{theo}\label{locck}
{\rm (1)} In connective $K$-theory, we have $\psi^v(w)=\theta_w \,K^Y(v,w)$.

{\rm (2)} Similarly, we have $\xi^v(w)=K^X(v,w)$.
\end{theo}

These statements can be made more explicit as follows, simply based on the definition of ${\mathcal R}^Y_w$ and ${\mathcal R}^X_w$, as well as on the relations between $Y_i$ and $X_i$. Recall that, in connective $K$-theory, the normalizing parameter $\theta_w\in S$ (defined in Lemma~\ref{evallem})  was calculated in \eqref{thetack}, so we can use this expression below. 

\begin{cor} 
{\rm (1)} We have
\begin{equation}\label{explicit1}
\psi^v(w)=\theta_w \sum_{1\le j_1<\ldots< j_k\le l}(-1)^k\, t^{k-\ell(v)} \,y_{\beta_{j_1}}\ldots y_{\beta_{j_k}}\,,
\end{equation}
where $\ell(v)\le k\le l$, and the summation ranges over the integer sequences $1\le j_1<\ldots< j_k\le l$ for which the corresponding subword $(p_1,\ldots,p_k)=(i_{j_1},\ldots,i_{j_k})$ of $I_w$ is subject to $Y_{p_1}\ldots Y_{p_k}=t^{k-\ell(v)}Y_v$.

{\rm (2)} We have
\begin{equation}\label{explicit2}
\xi^v(w)=\sum_{1\le j_1<\ldots< j_k\le l}(-t)^{k-\ell(v)} \, y_{-\beta_{j_1}}\ldots y_{-\beta_{j_k}}\,,
\end{equation}
where $\ell(v)\le k\le l$, and the summation ranges over the integer sequences $1\le j_1<\ldots< j_k\le l$ for which the corresponding subword $(p_1,\ldots,p_k)=(i_{j_1},\ldots,i_{j_k})$ of $I_w$ is subject to $X_{p_1}\ldots X_{p_k}=\pm t^{k-\ell(v)} X_v$.
\end{cor}

The formulas \eqref{billey-main}, \eqref{willems-main}, and \eqref{willems-main1} of Billey and Graham-Willems are clearly special cases of \eqref{explicit1} and \eqref{explicit2}, which can be retrieved by setting $t=0$ and $t=1$. We also need the conventions in Example~\ref{exfgl}~(1),(2), as well as the corresponding formulas for the normalizing parameter $\theta_w$; more precisely, as we saw in Examples~\ref{exeval}~(1),(2), we have $\theta_w=1$ in cohomology, and $\theta_w=e^{\rho-w\rho}$ in $K$-theory. 

We now prove Theorem \ref{locck} by extending the idea in \cite{Bi99}[Proof~of~Lemma~4.2]. Note that this proof is relatively short but not transparent, by contrast with the proof of the formulas of Billey, and Graham-Willems which we provide in Section \ref{bwrev}. Note also that the more consistent part of Billey's proof of \eqref{billey-main}, namely \cite{Bi99}[Theorem~2] on the independence of the corresponding root polynomial from a reduced word (which was proved using a case-by-case analysis), was already proved here as part of Theorem~\ref{mainthm}.

\begin{proof}[Proof of Theorem {\rm \ref{locck}}] We discuss only the first part, as the second one is completely similar. 
We proceed by decreasing induction on $\ell(v)$. The base case $v=w_\circ$ is easily verified, cf. Example~\ref{exbvv}~(1) and \eqref{topclass}. 

Assuming the result for $v$, and that $vs_i<v$ for some $i$, we will prove it for $vs_i$. Fix $w$ such that $ws_i<w$, so there exists a reduced word for $w$ which ends with $s_i$. By the definition of ${\mathcal R}^Y_w$, we have
\[{\mathcal R}^Y_w={\mathcal R}^Y_{ws_i}(1-y_{-w\alpha_i}Y_i)\,.\]
By the definition of $K^Y(u,w)$ and by using $Y_i^2=tY_i$, the above formula implies:
\begin{equation}\label{reckos} K^Y(u,w)=\casetwo{K^Y(u,ws_i)-y_{-w\alpha_i}\, t \,K^Y(u,ws_i)-y_{-w\alpha_i}\, K^Y(us_i,ws_i)}{u>us_i}{K^Y(u,ws_i)}{u<us_i}\end{equation}

We will now use \eqref{recpsixi}. On one hand, this implies $\psi^{v s_i}(w)=\psi^{v s_i}(ws_i)$. On another hand, by \eqref{reckos}, we have $K^Y(vs_i,w)=K^Y(vs_i,ws_i)$. Thus, we complete the induction step by calculating (cf.~\eqref{actiongkm}):
\begin{align*}(X_{-i}\,(\theta_{\cdot}\,K^Y(v,\cdot))(w)&=\frac{\theta_w\, K^Y(v,w)-\theta_{w s_i}\, K^Y(v, ws_i)}{-y_{-w\alpha_i}}\\
&=\theta_w \,\frac{K^Y(v,w)-(1-y_{-w \alpha_i}\, t) \, K^Y(v,ws_i)}{-y_{-w\alpha_i}}\\
&=\theta_w\, K^Y(vs_i,ws_i)\\
&=\theta_w \,K^Y(vs_i,w)\,.
  \end{align*}
Here the last two equalities follow directly from \eqref{reckos}; the second one uses the following fact:
\[\theta_w=\theta_{ws_i}\left(-\frac{y_{w\alpha_i}}{y_{-w\alpha_i}}\right)=\frac{\theta_{ws_i}}{1-y_{-w\alpha_i}t}\;,\]
which is based on the definition of $\theta_w$, on writing $y_{w\alpha_i}=\iota(y_{-w\alpha_i})$, and on the formula for $\iota(\,\cdot\,)$ in Example~\ref{exfgl}~(3).
\end{proof}

\subsection{The results of Billey and Graham-Willems revisited; duality in connective $K$-theory} \label{bwrev}

In this section we provide an alternative, transparent (and also uniform) proof of the results of Billey and Graham-Willems, which were recalled in Section~\ref{bwres}. 
Indeed, we show that modulo Theorem~\ref{mainthm} these are equivalent to the duality property \eqref{idealsh-dual}, which holds in both cohomology and $K$-theory. In fact, we prove this equivalence in the setup of connective $K$-theory developed in Section~\ref{seclocck}, and also derive the connective $K$-theory analogue of the mentioned duality property, which is yet missing. We also refer to the discussion about duality at the end of Section~\ref{fda}, in the setup of an arbitrary cohomology theory. 

\begin{theo}\label{duality}
The connective $K$-theory analogue of {\rm \eqref{idealsh-dual}} is equivalent to both parts of Theorem~{\rm \ref{locck}}.
\end{theo}

\begin{proof} We have the following equivalences in connective $K$-theory, where the statements are for all $v,w$:
\begin{equation}\label{cheq}
\theta_w\, K^Y(v,w)=\psi^v(w)\;\;\Longleftrightarrow\;\;b_{w,v}^Y=*\psi^v(w)\;\;\Longleftrightarrow\;\;Y_v^*(\delta_w)=*\psi^v(w)\,.
\end{equation}
The first statement is Theorem~{\rm \ref{locck}}~(1), the first equivalence is based on Theorem \ref{mainthm}, while the second one is immediate from the definition of the usual  dual $Y_{v}^* \in \mathbf{D}_F^\star$ of $Y_v$ and the expansion \eqref{expdw}. 
Now recall that, by \cite[Theorem~12.4]{CZZ1}, $Y_{v}^*$ is also Poincar\'e dual to $\zeta_v$. Thus, the statements in \eqref{cheq} are equivalent to $\zeta^v$ being Poincar\'e dual to $*\psi^v$, which is the connective $K$-theory analogue of \eqref{idealsh-dual}.

In the same way, we show that Theorem~{\rm \ref{locck}}~(2) is equivalent to $X_v^*(\delta_w)=*\xi^v(w)$. By \cite[Theorem~12.4]{CZZ1}, $X_{v}^*$ is  Poincar\'e dual to $\widetilde{\zeta}_v$. So the chain of equivalences continues with $*\xi^v$ being Poincar\'e dual to $\widetilde{\zeta}_v$ which, in turn, is equivalent to $(\xi^{vw_\circ})^\circ=\zeta_v$ being Poincar\'e dual to $*(\widetilde{\zeta}_{vw_\circ})^\circ=*\psi^v$; here we used the form of the pairing in~\cite{CZZ1} and \eqref{defpsiv}. This concludes the proof. 
\end{proof}

Summarizing the above arguments, based on the duality result of Kostant-Kumar \eqref{idealsh-dual} in cohomology and $K$-theory, we have a new proof of Theorem~{\rm \ref{locck}}, in the respective cases; furthermore, as we saw in Section~\ref{seclocck}, the explicit form of this theorem are the formulas \eqref{billey-main}, \eqref{willems-main}, and \eqref{willems-main1}, so these immediately follow. On the other hand, by reversing our reasoning, in connective $K$-theory we can rely on the proof of Theorem~{\rm \ref{locck}} in Section~\ref{seclocck}, and derive the analogue of \eqref{idealsh-dual} from the duality result for an arbitrary cohomology theory in \cite{CZZ1}.

\section{Bott-Samelson classes via root polynomials}\label{bscvrp}

In the previous section, we have considered two families of elements in $E_T^*(G/B)$, namely the Bott-Samelson classes ${\zeta}_{I_v}$ and $X_{I_v}^*$, both of which are viewed as functions from $W$ to $S$ (recall that $X_{I_v}^*$ are evaluated at $\delta_w$, cf. the expansion~\eqref{expdw}). By a completely similar reasoning to the proof of Theorem~\ref{duality}, we can see that these families essentially concide in $K$-theory, but not beyond that, due to the dependence on the reduced word $I_v$. Furthemore, in Theorem~\ref{mainthm} we have been able to express $X_{I_v}^*(\delta_w)$ in the hyperbolic case based on formal root polynomials, thus generalizing the results of Billey and Willems in ordinary cohomology and $K$-theory. We now suggest a similar approach for ${\zeta}_{I_v}$, which turns out to be a considerably harder problem. 

We start with no assumption on the underlying formal group law. The essence of the root polynomial approach to the calculation of the Schubert classes in ordinary cohomology and $K$-theory, as stated in Theorem~\ref{locck}, is to consider a certain ``generating function'' of these classes (cf. Proposition~\ref{wil-gen} and \eqref{zetaxi}), and express it based on an $X$-root polynomial. In general, we have Bott-Samelson classes, so we start by fixing an element $w\in W$, and reduced words $I_{vw_\circ}$, $I_{v}'$ for all $v\le w$. By analogy with the mentioned results in ordinary cohomology and $K$-theory, we formulate the following problem.

\begin{prob}\label{compgf} For suitable choices of $I_{vw_\circ}$ and $I_{v}'$ (to be specified), find a combinatorial formula for the generating function
\begin{equation}\label{gfell}
{\mathcal G}(w,(I_{vw_\circ},\,I_v')_{v\le w}):=\sum_{v\le w}\zeta_{I_{vw_\circ}}(ww_\circ)\, X_{I_v'}\,,
\end{equation}
as an element of $S\otimes_R Q_W$ (with $\zeta_{I_{vw_\circ}}(ww_\circ)$ as elements of the left tensor factor $S$).
\end{prob}

We are currently investigating this problem in the case of the hyperbolic formal group law \eqref{defell}, to which the rest of this section is devoted. We start by formulating the two conjectures below, based on our work so far and on computer tests. For simplicity, we denote the parameters in \eqref{defell} by $t:=\mu_1$ and $u:=\mu_2$, so the underlying ring $R$ is ${\mathbb Z}[t,u]$. We view the cases $t=0$ (i.e., the Lorentz formal group law), $t=1$ (see Remark~\ref{normalize}), and arbitrary $t$ as analogues of the ordinary cohomology, $K$-theory, and connective $K$-theory cases, respectively. Thus, it is helpful to make analogies with \eqref{billey-main}, \eqref{willems-main1}, and the discussion about connective $K$-theory at the end of Section~\ref{bwrev}, respectively. Indeed, the following conjecture would generalize these results, which are recovered by setting $u=0$. Recall from Theorem~\ref{mainthm}, cf. also Section~\ref{indredw}, that the formal root polynomials are independent of a reduced word in the hyperbolic case, so we can index them by the corresponding Weyl group element.

\begin{conj}\label{gfroot} We have, for suitable choices of $I_{vw_\circ}$ and $I_{v}'$:
\[{\mathcal G}(w,(I_{vw_\circ},\,I_v')_{v\le w})\,\in\,{\mathcal R}_w^X+u(S\otimes_R Q_W)\,.\]
\end{conj}

\begin{example} Consider the Lorentz formal group law, so $t=0$ and $X_i^2=0$. In the case of the root system $A_2$, let $w=w_\circ$ in $W=S_3$. For $v\ne w_\circ$, there is a unique reduced word $I_v=I_v'$, and we choose $I_{w_\circ}=I_{w_\circ}'=(1,2,1)$. We easily compute $\zeta_{I_{v}}({\rm Id})$ based on the setup in Section~\ref{gkmsetup}, see also \cite{GaRa}:
\begin{align*}&\zeta_\emptyset({\rm Id})=y_{-1}y_{-2}y_{-1-2}\,,\qquad\zeta_{(1)}({\rm Id})=y_{-2}y_{-1-2}\,,\qquad\zeta_{(2)}({\rm Id})=y_{-1}y_{-1-2}\,,\\
&\zeta_{(1,2)}({\rm Id})=\zeta_{(2,1)}({\rm Id})=y_{-1-2}\,,\qquad\zeta_{(1,2,1)}({\rm Id})=1+uy_{-2}y_{-1-2}\,.\end{align*}
Thus, we have
\begin{align*}{\mathcal G}(w_\circ,(I_{vw_\circ},\,I_v')_{v\in W})&=(1+y_{-1}X_1)(1+y_{-1-2}X_{2})(1+y_{-2}X_1)+uy_{-2}y_{-1-2}(1+y_{-1}X_1)\\
&={\mathcal R}_{w_\circ}^X+uy_{-2}y_{-1-2}(1+y_{-1}X_1)\,.\end{align*}
Note that, in order to get to the expression in the right-hand side, we need to rewrite $\zeta_{(2,1)}({\rm Id})$ (but not $\zeta_{(1,2)}({\rm Id})$) in the ring $S$ as follows:
\[\zeta_{(2,1)}({\rm Id})=y_{-1-2}=y_{-1}+y_{-2}+uy_{-1}y_{-2}y_{-1-2}\,.\]
So the Lorentz formal group law is involved in a crucial way.
\end{example}

Based on the analogies with \eqref{billey-main} and \eqref{willems-main1} mentioned above, as well as experimental evidence, we also formulate the following positivity conjecture. To state it, consider $v\le w$, and let $\Phi^+\cap w\Phi^-=\{\beta_1,\ldots,\beta_l\}$; also let $z_i=y_{-\beta_i}$. 

\begin{conj}\label{posconj} {\rm (1)} If $t=0$, $\zeta_{I_{vw_\circ}}(ww_\circ)$ is a sum with positive coefficients of terms of the form 
$$u^{(m-\ell(v))/2}z_{j_1}\ldots z_{j_m}\,,$$ 
where $\ell(v)\le m\le l$, $m-\ell(v)$ is even, and $1\le j_1<\ldots< j_m\le l$. 

{\rm (2)} If $t=1$, $\zeta_{I_{vw_\circ}}(ww_\circ)$ is a sum with positive coefficients of terms of the form 
$$(-1)^{k-\ell(v)}u^{(m-k)/2}z_{j_1}\ldots z_{j_m}\,,$$
 where $\ell(v)\le k\le m\le l$, $m-k$ is even, and $1\le j_1<\ldots< j_m\le l$. 
\end{conj}

This conjecture should be viewed in the context of Problem~\ref{compgf} and Conjecture~\ref{gfroot}. If true, it would be interesting to find the geometric reason for positivity.

\end{document}